\documentclass{amsart}

\usepackage[utf8]{inputenc}
\usepackage[T1]{fontenc}
\usepackage[english]{babel}
\usepackage{color}
\usepackage{amsmath} 
\usepackage{amssymb} 
\usepackage{amsthm}
\usepackage{graphicx}
\usepackage[colorlinks=true, linkcolor=blue]{hyperref}
\usepackage{cancel}

\usepackage{lmodern}
\usepackage{microtype}

\theoremstyle{plain}
\newtheorem{thm}{Theorem}[section]
\newtheorem*{thm*}{Theorem}
\newtheorem{prop}[thm]{Proposition}
\newtheorem{lem}[thm]{Lemma}
\newtheorem{cor}[thm]{Corollary}
\newtheorem*{cor*}{Corollary}

\newtheorem{defi}{Definition}[section]

\newtheorem{rem}{Remark}

\newcommand {\R} {\mathbb{R}} \newcommand {\Z} {\mathbb{Z}}
\newcommand {\T} {\mathbb{T}} \newcommand {\N} {\mathbb{N}}

\newcommand {\p} {\partial}
\newcommand {\dt} {\partial_t}

\DeclareMathOperator{\esssup}{ess-sup}

\begin{document}
\title[On the Boussinesq Equations]{On Enhanced Dissipation for the
  Boussinesq Equations}
\author{Christian Zillinger}
\address{BCAM -- Basque  Center  for  Applied  Mathematics, Mazarredo 14, E48009
  Bilbao, Basque Country -- Spain}
\email{czillinger@bcamath.org}
\begin{abstract}
  In this article we consider the stability and damping problem for the
  2D Boussinesq equations with partial dissipation near a two parameter family of
  stationary solutions which includes Couette flow and
  hydrostatic balance.

  In the first part we show that for the linearized problem in an infinite periodic channel
  the evolution is asymptotically stable if any diffusion coefficient is
  non-zero. In particular, this imposes weaker conditions than for example
  vertical diffusion.
  Furthermore, we study the interaction of shear flow, hydrostatic balance and
  partial dissipation.

  In a second part we adapt the methods used by Bedrossian, Vicol and Wang
  \cite{bedrossian2016sobolev} in the Navier-Stokes problem and combine them
  with cancellation properties of the Boussinesq equations to establish small
  data stability and enhanced dissipation results for the nonlinear Boussinesq problem with full dissipation.
\end{abstract}
\maketitle
\tableofcontents

\section{Introduction}
\label{sec:intro}

The Boussinesq equations are a common model in the study of heat
conduction and are given by a coupled system of the
Navier-Stokes equations and a diffusion equation for the temperature (see for
instance  \cite[Section 3.5]{temam2012infinite}).
In this article we specifically consider the two-dimensional incompressible
Boussinesq equations on $\T \times \R$ which model a heat-conducting
fluid in terms of its velocity field $v$, the pressure $p$ and its temperature density
$\theta$:
\begin{align*}
  \dt v + v \cdot \nabla v + \nabla p &= \nu_1 \p_x^2 v + \nu_2 \p_y^2 v +
  \begin{pmatrix}
    0 \\ \theta
  \end{pmatrix}, \\
  \dt \theta + v \cdot \nabla \theta &= \eta_1 \p_x^2 \theta + \eta_2 \p_y^2 \theta, \\
  \nabla \cdot v &=0, \\
  (t,x,y) & \in (0,\infty) \times \T \times \R.
\end{align*}
Here differences in $\theta$ cause the fluid to rise or fall due to buoyancy and the
temperature density is advected by the velocity.
The diffusion coefficients $\nu_1,\nu_2,\eta_1,\eta_2\geq 0$ are constants which model viscosity and thermal diffusion and may in general be anisotropic.

In the setting of full dissipation, that is if $\nu_1, \nu_2,\eta_1, \eta_2$
are all bounded below by a common constant $\nu>0$, global well-posedness
results are classical and make use of energy arguments (see for instance the
textbook by Teman \cite{temam2012infinite} or the articles
\cite{foias1987attractors,cannon1980initial}).
However, in some physical problems the thermal and viscous diffusivity $\nu_i, \eta_i$
may be of very different orders of magnitude or highly anisotropic.
In particular, some coefficients might be much smaller than all others.
A natural question thus concerns the problem of \emph{partial
  dissipation} where some of the coefficients are allowed to vanish.
Here, in a recent article Doering, Wu, Zhao and Zheng \cite{doering2018long}
consider the case without thermal diffusivity, $\nu_1=\nu_2>0$ and
$\eta_1=\eta_2=0$. We further mention the works by Titi, Lunasin and Larios,
\cite{li2016global,larios2013global} on vertical dissipation and
anisotropic dissipation and the works by Chae, Kim and Nam
\cite{chae1999local,chae2006global} on cases with no
viscosity or thermal diffusivity. For further discussion and references, the interested
reader is referred to the lecture notes by JH Wu \cite{wu20122d}.
In all these problems partial dissipation also implies a potential lack of smoothing and hence
questions of well-posedness become challenging problems. 

In this article we are interested in the behavior of the Boussinesq equations
with partial or full dissipation close to the following two parameter family of stationary solutions:
\begin{align}
\label{eq:1}
  v=
  \begin{pmatrix}
    \beta y \\ 0
  \end{pmatrix},\  \theta = \alpha y.
\end{align}
The case $\beta=0, \alpha>0$ is known as \emph{hydrostatic balance} and shares
structural similarities with stratified compressible flow (that is, with a mass
density $\rho$ instead of a temperature density $\theta$; see Section
\ref{sec:couette}).
The case $\beta=1, \alpha=0$ corresponds to a linear shear flow in the
Navier-Stokes equations (e.g. between moving plates or as a model for rotating
concentric cylinders) and is known as Couette flow.
We aim at understanding the asymptotic stability of
these solutions, the interaction of hydrostatic balance and shear and, in particular, at obtaining (mixing enhanced) dissipation
rates for the partial dissipation case.

Related settings have for instance been studied in the following works:
\begin{itemize}
\item Wu, Xu and Zhu \cite{wu2019stability} studied the nonlinear Boussinesq-B\'enard
  system near the trivial steady state $(0,0)$. 
\item The linearized inviscid Boussinesq problem near Couette flow ($\alpha=0$, $\beta=1$) was
  considered by Yang and Lin \cite{yang2018linear} in a work on stratified fluids
  (the linearized equations of the stratified fluids problem and of the
  Boussinesq problem are structurally similar).
\item In a recent work W. Tao and Wu \cite{tao20192d} consider the corresponding
  viscous linear problem with vertical dissipation, $\nu_x=\eta_x=0$, $\nu_y,
  \eta_y>0$ in the half-infinite periodic channel $\T \times (0,\infty)$.
  In Section \ref{sec:couette} we revisit this problem for the infinite channel
  $\T \times \R$ with general partial dissipation and with hydrostatic balance.
  Here the interaction of shear flow, hydrostatic balance and partial
  dissipation leads to challenging stability problems, while the absence of
  boundaries simplifies approaches by Fourier methods.
\item Since the Boussinesq equations are coupled Navier-Stokes equations,
  results for the latter are very closely related.
  In Section \ref{sec:fulldissipation} we adapt the strategy employed by Bedrossian, Vicol and Wang
\cite{bedrossian2016sobolev} for the 2D Navier-Stokes equations to the
Boussinesq equations. Furthermore, we combine these techniques with cancellation
properties of the Boussinesq equations to treat the case of ``large'' $\alpha$.
\end{itemize}

In this article we are interested in three main questions:
\begin{itemize}
\item How do shear flows, hydrostatic balance and diffusion interact and what
  (mixing enhanced) damping rates can be obtained?
\item How small should perturbations be so that that the nonlinear dynamics
  remain well-approximated by the linear dynamics?
  Or in other words, can we describe a Sobolev stability threshold for the
  nonlinear problem as the dissipation coefficients tend to zero?
\item How little dissipation is necessary for asymptotic stability results? In
  particular, we study how vanishing diffusivity coefficients effect decay rates
  and asymptotic stability results in the linearized problems.
\end{itemize}
We remark that in the inviscid case the linearized problem is algebraically
unstable at the level of the vorticity (see Lemma \ref{lem:trivialunstable}).
However, it is stable at the level of the velocity (see \cite{yang2018linear}).
In this work we focus on the (partially) viscous problem and stability of the
vorticity in Sobolev regularity.
In view of the results of Bedrossian, Vicol and Masmoudi
\cite{bedrossian2016enhanced} a further extension to the case of Gevrey regular
data seems possible but technically very challenging (see also the comments
following Corollary \ref{cor:Gevrey}). In particular, it would
have to precisely capture the growth and loss of regularity due to resonances
(see \cite{bedrossian2013landau,dengZ2019,dengmasmoudi2018,zillinger2020landau}).

\subsection{Main Results}
%\label{sec:main}

Our first main results concern small data nonlinear asymptotic stability and
enhanced dissipation for the setting with shear and with full dissipation
$\nu_x=\nu_y>0, \eta_x=\eta_y>0$.
In Theorem \ref{thm:nonlinearintro} (later restated as Theorem \ref{thm:nonlinear}) we focus on the setting where $\alpha$ is
``small'' and adapt the methods of \cite{bedrossian2016enhanced} used in the
Navier-Stokes problem near Couette flow to the Boussinesq equations.
We then combine these methods with energy arguments and cancellations for
hydrostatic balance (see \cite{doering2018long}) to treat the ``large'' $\alpha$
case in Theorem \ref{thm:largealphaintro} (later restated as Theorem \ref{thm:largealpha}).
Here and in the following results $\omega, v$ and $\theta$ denote the
perturbation of the two parameter family \eqref{eq:1} and if $\beta \neq 0$ we
consider coordinates moving with the shear:
\begin{align*}
  (x+ t \beta y, y, t).
\end{align*}
Under this change of variables the gradient and Laplacian are given by
\begin{align*}
  \nabla_t =
  \begin{pmatrix}
    \p_x \\ \p_t - t \beta \p_x
  \end{pmatrix},
  \Delta_t = \p_x^2 + (\p_y-t \beta \p_x)^2.
\end{align*}
\begin{thm}
  \label{thm:nonlinearintro}
  Let $N\geq 5$ and let $\beta=1$, $\epsilon_1\leq \frac{1}{100} \min(\nu,\eta)^{1/2}$,
  $\epsilon_{2} \leq \frac{1}{100}
  \sqrt{\eta}\sqrt{\nu} \epsilon_1$ and suppose that $0\leq \alpha < \eta^{1/2} \nu^{1/3}\frac{\epsilon_2}{\epsilon_1}$.
  Then if $\|\omega_0\|_{H^N} \leq \epsilon_1$ and $\|\theta_0\|_{H^N}\leq
  \epsilon_2$, the unique global solution with this initial data satisfies
  \begin{align}
    \|\omega\|_{L^\infty((0,\infty);H^N)}^2 + \nu \|\nabla_t \omega\|_{L^2((0,\infty); H^N)}^2 + \|\nabla_t \Delta_t^{-1} \omega\|_{L^2((0\infty); H^N)}^2  &\leq 8 \epsilon_1^2, \\
    \|\theta\|_{L^\infty((0,\infty); H^N)}^2 + \eta \|\nabla_t \theta\|_{L^2((0,\infty); H^N)}^2 &\leq 8 \epsilon^2.
  \end{align}
\end{thm}

\begin{thm}
  \label{thm:largealphaintro}
  Let $\alpha\geq 1$ and $\beta=1$ and $\nu>0$ and suppose that $\eta>2$.
  Let further $(\omega_0, \theta_0) \in H^N \times H^{N+1}$ be given initial
  data such that
  \begin{align*}
    \alpha\|\omega_0\|_{H^N}^2 + \|\nabla \theta_0\|_{H^N}^2 \leq \frac{1}{100} \epsilon^2.
  \end{align*}
  Then for all times $T>0$ it holds that
  \begin{align*}
    \esssup_{0\leq t \leq T} &(\alpha \|\omega(t) \|_{H^{N}}^2 +  \|\nabla_t \theta\|_{H^N}^2)\\
    &\quad + \nu \alpha \| \nabla_t \omega(t) \|_{L^2_t H^{N}}^2 \\
    &\quad + \eta \|\nabla_t \theta\|_{L^2_t H^N}^2 \leq \epsilon^2.
  \end{align*}
\end{thm}
For a discussion of the assumptions see Section \ref{sec:fulldissipation}.

As a corollary we derive exponential decay rates and enhanced dissipation (later
restated as Proposition \ref{prop:decayrates}).
\begin{prop}
  \label{prop:decayratesintro}
  Let $N, \alpha, \epsilon_1, \epsilon_2$ be as in Theorem \ref{thm:nonlinearintro}.
  Then the nonlinear Boussinesq equations further satisfy
  \begin{align*}
    \|\omega(t)\|_{H^N} + \|\theta(t)\|_{H^N} \leq 2 C \exp(- \min(\nu, \eta)^{1/3} t/10) (\|\omega_0\|_{H^N} + \|\theta_0\|_{H^N})
  \end{align*}
  for all $t>0$. In particular, we observe
  dissipation on a time scale $\min(\nu, \eta)^{-1/3}$ faster than heat flow. We
  say that the equations exhibit \emph{enhanced dissipation}.
\end{prop}

The nonlinear Boussinesq equations in particular with partial dissipation have
been studied in numerous previous works, e.g. \cite{larios2013global,li2016global,doering2018long} (see the introduction and Section
\ref{sec:fulldissipation} for a discussion). Our main differences and novelties here are:
\begin{itemize}
\item We consider the effects of a linear shear and hydrostatic balance at the
  same time. In particular, the effects of mixing by a shear flow and the resulting enhanced
  dissipation of the velocity field and the interaction of shear and hydrostatic
  balance have, to our knowledge, not previously been studied for the Boussinesq equations.
\item Our results concern higher regularity and decay rates near combinations of
  shear flow and hydrostatic balance.
  In contrast, well-posedness and asymptotic stability results such as
  \cite{doering2018long,larios2013global,chae1999local}
   focus on perturbations of $(0,0)$ and energies at the level of $L^2 \times
   H^1$ or make use of energy functionals of the type $\alpha \|v\|_{L^2}^2 +
   \|\theta\|_{L^2}^2$ (for hydrostatic balance \emph{without} shear).
\item In particular, our results are stable under the limit $\alpha\downarrow 0$
  and incorporate mixing enhanced dissipation rates.
\end{itemize}

In addition to the nonlinear results obtained in Theorems \ref{thm:nonlinearintro}
and \ref{thm:largealphaintro} we
also study the linearized setting around more general profiles of the form
\eqref{eq:1}, derive finer characterizations of asymptotics and are in
particular interested in the effects of partial dissipation.

More precisely, we ask how little dissipation is required for asymptotic
stability of the vorticity and the temperature to hold and how this is
influenced by shear and hydrostatic balance, respectively.
In this context we mention numerous previous works by J. Wu and coauthors on
related (sub)settings \cite{tao20192d,wu2019stability} (see Section
\ref{sec:couette} for a longer discussion). 
Our main results are collected in Theorem \ref{thm:summarylinear}, which we
restate here, and are derived in the
subsections of Section \ref{sec:couette}.

We recall that for $\beta\neq 0$ we consider the perturbations in coordinates
moving with the shear flow:
\begin{align*}
  \omega(t, x+ \beta t y, y), \theta(t, x+t \beta y, y).
\end{align*}
\begin{thm}
  \label{thm:summarylinearintro}
  Consider the linearized problem \eqref{eq:linear} around the state
  \begin{align*}
    v^*=(\beta y,0), \ \theta^* = \alpha y,
  \end{align*}
  with initial data $\omega_0 \in H^N$ and $\theta_0 \in H^{N+1}$, $N \in \N$.

  In the inviscid case the evolution of the vorticity $\omega$ is unstable in
  the sense that
  \begin{align*}
    \limsup_{t\rightarrow \infty} \|\omega(t)\|_{H^N}=\infty,
  \end{align*}
  unless $\alpha>0, \beta=0$ or $\alpha=0$ and $\p_x \theta$ is trivial.

  If $\alpha=0$ the evolution of the vorticity is asymptotically stable if at
  least one diffusion coefficient is non-zero. More precisely, for every $N \in
  \N$ there exists $C=C(\eta_x, \eta_y, N)$ such that the temperature
  density satisfies
  \begin{align*}
    \|\theta(t)-\int \theta(t)dx \|_{H^N} \leq C \exp\left(-\eta_x t - \eta_y \frac{t^3}{12}\right) \|\theta_0\|_{H^N}.
  \end{align*}
  Furthermore, there exists $C=C(\eta_x, \eta_y, \nu_x, \nu_y,N)$ and a profile $\omega_1(t)$ (see Theorem
  \ref{thm:stabledamping} for a detailed description) such that
  \begin{align*}
    \|\omega(t)-\omega_1(t)\|_{H^N}\leq C \exp\left(-\max(\eta_y,\nu_y)\frac{t^3}{12} - \max(\eta_x, \nu_x)t\right) (\|\omega_0\|_{H^N} + \|\p_x \theta_0\|_{H^N})
  \end{align*}
  and $\omega_1(t)$ satisfies
  \begin{align*}
    \|\omega_1(t)\|_{H^N} \leq C(\nu_x,\nu_y,\eta_x,\eta_y) \exp(-\min(\eta_y,\nu_y)\frac{t^3}{12} - \min(\eta_x, \nu_x)t/2) (\|\omega_0\|_{H^N} + \|\p_x \theta_0\|_{H^N}).
  \end{align*}

  Finally, let $\alpha>0$ and suppose that at least one of $\min(\nu_x, \eta_x)$
  or $\min(\nu_y, \eta_y)$ is positive. Then it holds that
  (see Propositions \ref{prop:woshear} and \ref{prop:wshear})
  \begin{align*}
   E(t):= \alpha \|\omega(t)\|_{H^N}^2 + \|\p_x \theta(t) \|_{H^N}^2 + \|(\p_y-t \beta \p_x) \theta(t)\|_{H^N}^2.
  \end{align*}
  is bounded uniformly in time. Furthermore, if $\beta=1$ (or more generally
  $\beta\neq 0$) we obtain the enhanced dissipation estimates:
  \begin{align}
    \label{eq:28}
   E(t)\leq C (1+t^2) \exp\left(-\min(\nu_y, \eta_y)\frac{t^3}{12}- \min(\eta_x, \nu_x)t\right) E(0).
  \end{align}
\end{thm}
We remark that in the inviscid problem a natural regularity class is given by the
Gevrey class $\mathcal{G}_2$ (see \cite{dengZ2019,zillinger2019linear,jia2019linear,dengmasmoudi2018}).
As observed in \cite{zillinger2019linear} stability in Gevrey classes can be
derived as a corollary of quantitative control in Sobolev spaces.
For simplicity of notation and as an example we state such a corollary for the
case $\alpha>0, \beta=0$ of Theorem \ref{thm:summarylinearintro} (see also Proposition \ref{prop:woshear}).
\begin{cor}
\label{cor:Gevrey}
  Let $\alpha>0$, $\beta=0$ and let $\nu_x, \eta_x,\nu_y, \eta_y\geq 0$ be
  given.
  Suppose that $\omega_0, \theta_0$ are in the Gevrey class $\mathcal{G}_2$,
  that is there exists $R_1, R_2>0$ such that for all $j \in \N$
  \begin{align*}
    \|\omega_0\|_{H^j}\leq R_2^{1+j} (j!)^{2j}, \\
    \|\nabla \theta_0\|_{H^j}\leq R_2^{1+j} (j!)^{2j}.
  \end{align*}
  Then there exist $R_3, R_4>0$, which depend on $R_1, R_2$ and $\alpha$ such
  that the solution of the linearized Boussinesq equations $(\omega,\theta)$
  satisfies 
    \begin{align*}
    \|\omega(t)\|_{H^j}\leq R_3^{1+j} (j!)^{2j}, \\
    \|\nabla \theta(t)\|_{H^j}\leq R_4^{1+j} (j!)^{2j}.
  \end{align*}
  for all times $t>0$ and all $j \in \N$.
\end{cor}

\begin{proof}[Proof of Corollary \ref{cor:Gevrey}]
  By Theorem \ref{thm:summarylinearintro} we know that
  \begin{align*}
    & \quad \alpha \|\omega(t)\|_{H^N}^2 + \|\nabla \omega(t)\|_{H^N}^2 \\
    &\leq c(t) (\alpha \|\omega_0\|_{H^N}^2 + \|\nabla \omega_0\|_{H^N}^2)\\
    &\leq c(t) (\alpha R_1^{2N} + R_2^{2N})(j!)^{4},
  \end{align*}
  where
  \begin{align*}
    c(t)= \exp(-\min(\nu_x, \eta_x)t - (\min(\nu_y,\eta_y)\frac{t^3}{8})) \leq 1.
  \end{align*}
  Thus we may choose
  \begin{align*}
    R_3&\geq \sqrt[N]{c(t)(R_1^{2N} + \alpha^{-1}R_2^{2N})}, \\
    R_4&\geq \sqrt[N]{c(t)(\alpha R_1^{2N} +R_2^{2N})},
  \end{align*}
  which can be controlled in terms of $2\max(\alpha, \alpha^{-1})\max(R_1,R_2)$.
\end{proof}
  We remark that more generally it suffices to establish a bound of the form
  \begin{align*}
    \|(\omega(t), \theta(t))\|_{H^N \times H^{N+1}}\leq c^{N} \|(\omega_0, \theta_0)\|_{H^N \times H^{N+1}},
  \end{align*}
  for some constant $c$ independent of $N$ (see Section \ref{sec:couette} for
  several estimates of this type).
  Furthermore, due to the change of coordinates associated with $\beta$, in the general
  case one may obtain estimates of the form
  \begin{align*}
    \|\theta(t)\|_{H^N}\leq \|\nabla_t \theta\|_{H^N}\leq c^N \|\nabla \theta_{0}\|_{H^N}\leq c^N \|\theta_0\|_{H^{N+1}},
  \end{align*}
  which ``lose'' one derivative. For this reason general Gevrey estimates either
  need to track spaces more precisely or allow for losses in $R$ or the Gevrey
  class exponent with time. As this is not a focus of the article, we opted to
  only state a simple result.

The remainder of the article is organized as follows:
\begin{itemize}
\item In Section \ref{sec:notation} we introduce notational conventions used
  throughout the article.
\item In Section \ref{sec:couette} we consider the linearized problem around
  the two parameter family \eqref{eq:1}. Here a particular focus is placed on the
  problem of partial dissipation and we show that if even just one dissipation
  coefficient is non-trivial asymptotic stability results hold.
  Furthermore we discuss how the interaction of shear flow and hydrostatic
  balane influence (mixing enhanced) dissipation rates.
\item In Section \ref{sec:fulldissipation} we discuss the nonlinear problem with
  full dissipation. In a first result we adapt the approach Bedrossian, Vicol and Wang
  \cite{bedrossian2016sobolev} used for the Navier-Stokes problem to the
  Boussinesq equations and establish stability in Sobolev regularity for small
  data and small slope $\alpha\geq 0$ of the hydrostatic balance. We then
  combine these tools with additional cancellation properties of the Boussinesq
  equations with hydrostatic balance (see \cite{doering2018long}) to treat the
  case of ``large'' $\alpha$. 
\item In Section \ref{sec:bounds} we show that the nonlinear stability results
  of Section \ref{sec:fulldissipation} combined with the estimates on the linear
  problem obtained in Section \ref{sec:couette} yield nonlinear (enhanced)
  dissipation estimates.
\end{itemize}

\subsection{Notation}
\label{sec:notation}
In the study of both the linearized and nonlinear Boussinesq equations we make
extensive use of the Fourier transform. We denote the Fourier transform of a
function $f(x,y) \in L^2 (\T \times \R)$ by
\begin{align*}
\tilde{f}(k,\xi):= (\mathcal{F}_{x,y}f)(k,\xi)
\end{align*}
with $k \in \Z$ being discrete and $\xi \in \R$.

In our analysis the $x$-average, $k=0$, plays a distinct role in that it might be conserved
or decay slower than its $L^2$-orthogonal complement.
We thus denote
\begin{align*}
  f_=(y)= \int_{0}^1 f(x,y)dx
\end{align*}
and its complement
\begin{align*}
  f_{\neq}(x,y)= f(x,y)-f_{=}(y).
\end{align*}
As related notation in Section \ref{sec:fulldissipation} we split a nonlinear
sum of integrals $\mathcal{T}$ into contributions $\mathcal{T}^=$ involving the
$x$-average of the velocity (which is a shear flow) and its complement
$\mathcal{T}^{\neq}$.

In this article our main object of interest is the evolution of perturbations
around the stationary states given by the two parameter family \eqref{eq:1}.
Hence, with slight abuse of notation we use $\omega, v$ and $\theta$ to refer to
the \emph{perturbation} of the vorticity, velocity and temperature (instead of the full
solution).
Similarly, when $\beta \neq 0$ it is natural to work in coordinates moving with
the flow and consider
\begin{align*}
  \omega(t, x+t \beta y, y), v(t, x+t \beta y, y), \theta(t, x+t \beta y, y),
\end{align*}
as well as Sobolev spaces with respect to these coordinates.

In Section \ref{sec:fulldissipation} we consider spaces of the form $L^p((0,T);
H^N)$ or $L^p((0,\infty), H^N)$, which we abbreviate as $L^p_tH^N$.
In some asymptotic estimates we denote universal constants, which do not depend
on the quantities under consideration, by $C>0$. These constants may change from
line to line.
Similarly, we write $a \ll b$ if there exists a small universal constant ($C\leq
\frac{1}{100}$ for our purposes) such that $|a|\leq C |b|$.

\section{The Linearized Problem around Couette Flow and Hydrostatic Balance}
\label{sec:couette}
In this section we consider the linearized two-dimensional Boussinesq equations
on $\T \times \R$ near the two-parameter family of stationary solutions
\begin{align*}
  v=(\beta y, 0),\  \theta=\alpha y
\end{align*}
and with possibly partial dissipation:
\begin{align}
  \label{eq:linear}
  \begin{split}
  \dt \omega + \beta y \p_x \omega &= \nu_x \p_{xx} \omega + \nu_y \p_{yy} \omega+ \p_x \theta, \\
  \dt \theta + \beta y \p_x \theta + \alpha \p_x \Delta^{-1} \omega &= \eta_x \p_{xx} \theta + \eta_y \p_{yy} \theta,\\
  (t,x,y) &\in \R \times \T \times \R, \\
  \nu_x, \nu_y, \eta_x, \eta_y&\geq 0, \\
  \alpha\geq 0, \beta &\in \R.
  \end{split}
\end{align}
In a recent work L. Tao and J. Wu \cite{tao20192d} considered the related
(sub)case of the linearization of
the Boussinesq equations with vertical dissipation in both vorticity and temperature
\begin{align*}
   \eta_x=0, \eta_y>0, \nu_x=0,\nu_y>0,
\end{align*}
for $\beta=1, \alpha=0$ in the periodic half-space $\T \times (0,\infty)$ with
Neumann boundary conditions.

In our setting, on the one hand, the interaction of non-trivial shear $\beta\neq 0$ and
non-trivial balance $\alpha>0$ and allowing for more diffusion coefficients to
vanish allows for a multitude of different dynamics and stability results and
proves very challenging in the case of full generality.
On the one hand, as this setting does not possess boundaries, Fourier methods can be more easily
used and allow for a fine, optimal descriptions of asymptotic behavior.
In particular, we can clearly isolate the effects of each diffusion parameter
and show that in this setting it suffices to impose even weaker conditions on the
diffusivity parameters than in \cite{doering2018long} or
\cite{tao20192d}: Only a \emph{single} parameter needs to be non-zero.

\begin{thm}
  \label{thm:summarylinear}
  Consider the linearized problem \eqref{eq:linear} around the state
  \begin{align*}
    v=(\beta y,0), \ \theta = \alpha y,
  \end{align*}
  with initial data $\omega_0 \in H^N$ and $\theta_0 \in H^{N+1}$, $N \in \N$.

  In the inviscid case the evolution of the vorticity $\omega$ is unstable in
  the sense that
  \begin{align*}
    \limsup_{t\rightarrow \infty} \|\omega_0\|_{H^N}=\infty,
  \end{align*}
  unless $\alpha>0, \beta=0$ or $\alpha=0$ and $\p_x \theta$ is trivial.

  If $\alpha=0$ the evolution of the vorticity is asymptotically stable if at
  least one diffusion coefficient is non-zero. More precisely, for every $N \in
  \N$ there exists $C=C(\eta_x, \eta_y, N)$ such that the temperature
  density satisfies
  \begin{align*}
    \|\theta_{\neq}(t)\|_{H^N} \leq \exp\left(-\eta_x t - \eta_y \frac{t^3}{12}\right) \|\theta_0\|_{H^N}
  \end{align*}
  Furthermore, there exists $C=C(\eta_x, \eta_y, \nu_x, \nu_y,N)$ and a profile $\omega_1(t)$ (see Theorem
  \ref{thm:stabledamping}) such that
  \begin{align*}
    \|\omega(t)-\omega_1(t)\|_{H^N}\leq C \exp\left(-\max(\eta_y,\nu_y)\frac{t^3}{12} - \max(\eta_x, \nu_x)t\right) (\|\omega_0\|_{H^N} + \|\p_x \theta_0\|_{H^N})
  \end{align*}
  and $\omega_1(t)$ satisfies
  \begin{align*}
    \|\omega_1(t)\|_{H^N} \leq C(\nu_x,\nu_y,\eta_x,\eta_y) \exp\left(-\min(\eta_y,\nu_y)\frac{t^3}{12} - \min(\eta_x, \nu_x)t/2\right) (\|\omega_0\|_{H^N} + \|\p_x \theta_0\|_{H^N}).
  \end{align*}

  Finally, let $\alpha>0$ and suppose that at least one of $\min(\nu_x, \eta_x)$
  or $\min(\nu_y, \eta_y)$ is positive. Then it holds that
  (see Propositions \ref{prop:woshear} and \ref{prop:wshear})
  \begin{align*}
    E(t):=\alpha \|\omega(t)\|_{H^N}^2 + \|\p_x \theta(t) \|_{H^N}^2 + \|(\p_y-t \beta \p_x) \theta(t)\|_{H^N}^2.
  \end{align*}
is bounded uniformly in time. Furthermore, if $\beta=1$ (or more generally
  $\beta\neq 0$) we obtain the enhanced dissipation estimates:
  \begin{align*}
   E(t)\leq C (1+t^2) \exp\left(-\min(\nu_y, \eta_y)\frac{t^3}{12}- \min(\eta_x, \nu_x)t\right) E(0).
  \end{align*}
\end{thm}

In order to introduce methods and techniques, we first consider some exceptional cases, such as
$\alpha=0$ in Sections \ref{sec:inviscid} and \ref{sec:homogen} and $\alpha>0, \beta=0$ in Section
\ref{sec:woshear}.
The setting with both effects $\alpha>0, \beta =1$ is then considered in Section
\ref{sec:vert}.
Finally, we revisit these results in Section \ref{sec:bounds} to establish decay
rates for the nonlinear problem with small data.

\subsection{The Inviscid Case}
\label{sec:inviscid}

In this section we consider the inviscid problem
$\eta_x=\eta_y=\nu_x=\nu_y=0$ with $\alpha\geq 0$ and $\beta \in \R$ to
study the interaction between shear and hydrostatic balance.

As a first simple model setting we consider the case of homogeneous temperature,
($\alpha=0$) and an affine flow ($\beta \in \R$).
Here we obtain a simple, explicit solution and in particular observe that the evolution is linearly algebraically unstable at the level of the vorticity
but the density $\theta$ is stable, as is the velocity.
\begin{lem}
  \label{lem:trivialunstable}
  Consider the inviscid linearized problem \eqref{eq:linear} with $\alpha=0$ and
  $\beta \in \R$ on $\T \times \R$ (or $\T \times I$)
  \begin{align*}
      \dt \omega + \beta y \p_x \omega &= \p_x \theta, \\
    \dt \theta + \beta y \p_x \theta&= 0,\\
  \omega|_{t=0}=\omega_0, \theta|_{t=0}&=\theta_0,
  \end{align*}
  with initial data
  \begin{align*}
    (\omega, \eta)|_{t=0}=(\omega_0, \theta) \in H^{N} \times H^{N+1}. 
  \end{align*}
  It has the following explicit solution:
  \begin{align*}
    \omega(t,x,y)&= \omega_0(x-\beta ty,y) +t \p_x\theta_0(x- \beta ty,y), \\
    \theta(t,x,y)&= \theta_0(x-\beta ty,y).
  \end{align*}
  In particular, $\theta(t,x+ \beta ty,y)$ is stationary and hence stable and
  the velocity field satisfies
  \begin{align*}
    \|v(t,x,y)\|_{L^2}\leq \|\omega_0\|_{L^2} + C \frac{1}{\beta} \|\nabla \theta_0\|_{L^2}.
  \end{align*}
  The velocity is stable stable in $L^2$ for any $\beta \neq 0$. 

  However, the evolution of $\omega(t,x,y)$ and $\omega(t,x+\beta ty,y)$ is
  unstable in any positive Sobolev norm unless $\p_x\theta_0$ is trivial.
\end{lem}
We interpret this to say that a shear $\beta y$ has a stabilizing effect on the
velocity and that $\p_x \theta_0$ has a destabilizing effect on $\omega$.
Such a stability result for the velocity has previously been obtained by Lin and Yang
\cite{yang2018linear} in a work on the linearized inviscid, stratified Euler equations
around $u=(y,0),\rho=e^{-\gamma y}$ (which yield a very similar equation).
However, in view of the nonlinear problem considered in Section
\ref{sec:fulldissipation} we here emphasize the instability of the
vorticity due to $\p_x \theta_0$.
Our question in the following is then how much dissipation is required to
restore asymptotic stability of the vorticity (see Theorem \ref{thm:stabledamping}).

\begin{proof}[Proof of Lemma \ref{lem:trivialunstable}]
  In the Lagrangian coordinates $(x-ty,y)$ the system reads
  \begin{align*}
    \frac{d}{dt} \omega &= \p_x \theta, \\
    \frac{d}{dt} \theta &=0.
  \end{align*}
  One observes that the explicit solution of this system is given by
  \begin{align*}
    \theta&=\theta_0, \\
    \omega&= \omega_0 + t \p_x \theta_0.
  \end{align*}
  The result of the lemma then follows by expressing these solutions in Eulerian
  coordinates.
  Concerning the stability estimate of the velocity, we note that
  \begin{align*}
    t\p_x \theta_0(x-ty,y) = (-\frac{d}{dy}+\p_y) \theta_0(x-ty,y).
  \end{align*}
  Since the velocity corresponds to gaining one derivative compared to the
  vorticity, we may thus absorb the $\frac{d}{dy}$ and hence obtain a uniform
  bound.
  However, we remark that while $\omega$ only depends on $\p_x \theta_0$, not
  the full gradient, in this estimate of the velocity we require control of
  $\p_y \theta_0$ as well.
\end{proof}

In the following lemma we consider the effect of affine hydrostatic balance
$\alpha > 0$. The positive sign here corresponds to hotter fluid being on top.
If this is inverted the solution is known to be unstable \cite[Theorem 1.4
(3)]{doering2018long}.
Here, if there is no shear ($\beta=0$) the hydrostatic balance serves to
stabilize the dynamics of the vorticity.
However, if $\beta\neq 0$ the evolution of the vorticity is still algebraically
unstable with a rate depending on $\alpha$ and $\beta$. 
 \begin{lem}
   Consider the inviscid problem \eqref{eq:linear} with $\alpha>0$ and initial
   data
   \begin{align*}
     (\omega_0, \theta_0) \in H^N \times H^{N+1}.
   \end{align*}
  If there is no shear, $\beta=0$, then the evolution
  \begin{align*}
    (\omega_0,\theta_0) \mapsto (\omega(t), \theta(t))
  \end{align*}
  is stable as a map on $H^N\times H^{N+1}$ for any $N\geq 0$.
  More precisely, for every $\alpha>0$ and every $t\geq 0$ it holds that 
  \begin{align*}
    \|\omega(t)\|_{H^N} &\leq \|\omega_0\|_{H^N} + \frac{1}{\sqrt{\alpha}} \|\nabla \theta_0\|_{H^N},\\ 
    \|\theta(t)\|_{H^{N+1}} &\leq \sqrt{\alpha} \|\omega_0\|_{H^N} + \|\nabla \theta_0\|_{H^{N+1}}. 
  \end{align*}
  
  If there is shear, $\beta\neq 0$, then the evolution of
  \begin{align*}
    \omega(t,x+\beta t y, y), \theta(t,x+\beta t y, y)
  \end{align*}
  is unstable in $H^N\times H^{N+1}$ with an algebraic growth rate $t^{\gamma}$
  as $t\rightarrow \infty$. Here $\gamma$ depends on $\alpha$ and $\beta$.
\end{lem}

\begin{proof}
  \underline{The case without shear:}
  In the case $\beta=0$ the equation reduces to
  \begin{align*}
    \dt \omega &= \p_x \theta, \\
    \dt \theta &= \alpha \p_x \Delta^{-1} \omega.
  \end{align*}
  Taking a Fourier transform in both $x$ and $y$ we obtain a two-dimensional
  constant coefficient ODE system at each frequency:
  \begin{align*}
    \dt
    \begin{pmatrix}
      \tilde{\omega}\\
      \tilde{\theta}
    \end{pmatrix}
=
    \begin{pmatrix}
      0 & ik \\
      i \alpha \frac{k}{k^2+\xi^2}  & 0
    \end{pmatrix}
          \begin{pmatrix}
      \tilde{\omega}\\
      \tilde{\theta}
    \end{pmatrix}=
    \begin{pmatrix}
      0 \\ 0
    \end{pmatrix}.
  \end{align*}
  This then has the explicit solution
  \begin{align*}
        \begin{pmatrix}
      \tilde{\omega}\\
      \tilde{\theta}
    \end{pmatrix} (t)= 
    \begin{pmatrix}
      \cos(\frac{k\sqrt{\alpha}}{\sqrt{k^2+\xi^2}}t) &\frac{i \sqrt{k^2+\xi^2}}{\sqrt{\alpha}} \sin(\frac{k\sqrt{\alpha}}{\sqrt{k^2+\xi^2}}t) \\
    \frac{i \sqrt{\alpha}}{\sqrt{k^2+\xi^2}} \sin(\frac{k\sqrt{\alpha}}{\sqrt{k^2+\xi^2}}t)   & \cos(\frac{k\sqrt{\alpha}}{\sqrt{k^2+\xi^2}}t)
 \end{pmatrix}
               \begin{pmatrix}
      \tilde{\omega}_0\\
      \tilde{\theta}_0
    \end{pmatrix}.
  \end{align*}
    In particular, we observe that $\sqrt{k^2+\xi^2} \tilde{\theta}_0$ loses one
    derivative in $x$ \emph{and} $y$ as opposed to just $\p_x \theta_0$ in the
    $\alpha=0$ case. In contrast $\frac{\sqrt{\alpha}}{\sqrt{k^2+\xi^2}}$ gains
    one derivative.

    We remark that as $\alpha\downarrow 0$ we recover the
    growth by $t$ as in Lemma \ref{lem:trivialunstable}.

\underline{The case with shear:} If $\beta\neq 0$ we may consider a rescaling of
$y \mapsto \beta y$ to obtain:
\begin{align*}
  \dt \omega + y \p_x \omega &= \p_x \theta, \\
  \dt \theta + y \p_x \theta+ \alpha \p_x (\p_x^2 + \beta^2 \p_y^2)^{-1} \omega&=0
\end{align*}
In view of stability properties of the flow by $y\p_x$ we further change to
coordinates $(x+ty,y)$ (or $(x+\beta ty, y)$ in the original coordinates).
In these coordinates a Fourier transform then leads to the following
\emph{time-dependent} ODE system:
    \begin{align*}
    \dt
    \begin{pmatrix}
      \tilde{\omega}\\
      \tilde{\theta}
    \end{pmatrix}
=
    \begin{pmatrix}
      0 & ik \\
      i \alpha \frac{k}{k^2+\beta^2(\xi-kt)^2}  & 0
    \end{pmatrix}
          \begin{pmatrix}
      \tilde{\omega}\\
      \tilde{\theta}
    \end{pmatrix}.
    \end{align*}
    We note that for $k=0$ this system is trivial. In the following thus let
    $k\neq 0$ be arbitrary but fixed.

    As the matrix is time-dependent, we cannot anymore use a matrix exponential
    function to solve it. Instead we follow a similar approach as in a prior
    work on fluid echoes in Euler's equations \cite{dengZ2019} and consider a corresponding second
    order ODE system.
    Indeed, since $ik$ does not depend on $t$ we observe that the equation for $\p_t^2 \tilde{\omega}$ decouples and is given by a Schrödinger problem with
    potential:
    \begin{align}
      \label{eq:2}
      \dt^2 \tilde{\omega} + \frac{\alpha}{\beta^2} \frac{k^2}{k^2+(\xi-kt)^2} \tilde{\omega}=0.
    \end{align}
    We remark that we may recover
    \begin{align*}
    \tilde{\theta}=\frac{1}{ik} \p_t\tilde{\omega}.
    \end{align*}
    in terms of $\p_t \omega$. Thus it suffices to understand how $\omega$ and
    $\p_t \omega$ evolve under the equation \eqref{eq:2}.

    Shifting in time by $\frac{\xi}{k}$, problem \eqref{eq:2} becomes independent of $k$:
    \begin{align*}
       \dt^2 \tilde{\omega} + \alpha \frac{1}{1+(\beta t)^2} \tilde{\omega}=0.
    \end{align*}
    We then further rescale time by $t \mapsto \frac{1}{\beta}t=:\tau$, which
    yields
    \begin{align*}
      \p_\tau^{2} \tilde{\omega} + \alpha \beta^2\frac{1}{1+\tau^2} \tilde{\omega}=0.
    \end{align*}
    For simplicity of notation in the following we consider the special case
    $\beta=1$ and again use $t$ for the time variable. However, by the above
    scaling argument this is no loss of generality.
    
    This problem then has an explicit solution in terms of hypergeometric
    functions of the second kind (see the NIST Digital Library of Mathematical
    Functions \cite{NIST:DLMF}, Chapter 15) :
    \begin{align*}
      \omega&= C_1 F(-\frac{1}{4}-\frac{1}{4}\sqrt{1-4\alpha},-\frac{1}{4}+\frac{1}{4}\sqrt{1-4\alpha}, \frac{1}{2},-t^2) \\
      & \quad + C_2 t F(\frac{1}{4}-\frac{1}{4}\sqrt{1-4\alpha},+\frac{1}{4}-\frac{1}{4}\sqrt{1-4\alpha}, \frac{3}{2},-t^2).
    \end{align*}
    In particular, we note that asymptotically (see Chapter 15.8 in \cite{NIST:DLMF}) 
    \begin{align*}
      F(a,b,c;z) \sim c_1 z^{-a}(1+\mathcal{O}(z^{-1})) + c_2 z^{-b}(1+\mathcal{O}(z^{-1})).
    \end{align*}
    as $z=-t^2$ tends to $-\infty$.
    Since
    \begin{align*}
      -a=\frac{1}{4}+\frac{1}{4}\sqrt{1-4\alpha}
    \end{align*}
    has positive real part for $\alpha\neq 0$ we conclude that the
    evolution for $\omega$ is algebraically unstable.
\end{proof}
In these introductory results we have seen that $\beta \neq 0$ and $\alpha>0$
introduce competing (de)stabilizing effects and that the evolution of the
vorticity in the inviscid problem is generally unstable. In the following we
investigate whether stability can be restored by dissipation and if so how much
dissipation is required. Here we first consider the case $\alpha=0$ in Section
\ref{sec:homogen} and then $\alpha>0$ in Section \ref{sec:alpha}.

\subsection{The Homogeneous, Partial Dissipation Case}
\label{sec:homogen}

In this section we consider the problem of homogeneous hydrostatic balance,
$\alpha=0$, and shear flow, $\beta \in \R$, with partial dissipation. The case of affine
balance, $\alpha>0$, is studied in Section \ref{sec:alpha}.
Due to the constant coefficient structure and the absence of boundary terms we
here can construct (semi-)explicit solutions and thus clearly identify the
effects of each dissipation coefficient.

Problems of partial dissipation naturally appear as limiting cases where for
instance vertical and horizontal length scales are of very different magnitude
or either thermal or viscous effects are considered dominant.
In particular, we mention the work of Doering, Wu, Zhao and Zheng
\cite{doering2018long} on the nonlinear problem without buoyancy diffusion
($\alpha>0$, $\beta=0$, $\eta_x=\eta_y=0$, $\nu_x,\nu_y>0$) and the work by L. Tao and Wu
\cite{tao20192d} on the linearized problem with shear and vertical diffusion
($\alpha=0$, $\beta=1$, 
$\eta_x=\nu_x=0$, $\eta_y,\nu_y>0$).

In the following we consider the linear problem \eqref{eq:linear} with
$\alpha=0$ and in particular show that if at least just \emph{one} of the diffusivity
coefficients is positive then the problem is asymptotically
stable.
Moreover, if at least one of $\eta_y$ and $\nu_y$ is positive the problem
exhibits enhanced dissipation, that is damping on faster time scales than might
be expected for heat flow.
Thus, in this setting we can hence show directly that milder assumptions are
sufficient.

\begin{thm}
  \label{thm:stabledamping}
 Consider the linearized Boussinesq problem \eqref{eq:linear} for $\alpha=0,
 \beta=1$ with $\omega_0,
 \theta_0,\p_x \theta_0 \in H^N$, $N \geq 0$, and suppose that at least one of
 $\eta_x,\eta_y,\nu_x,\nu_y\geq 0$ is non-trivial.
 Then the evolution is asymptotically stable in the following sense.

 The $x$-averages $\omega_{=}(t,y), \theta_{=}(t,y)$ (see Section
 \ref{sec:notation} for a summary of notation) satisfy the
 one-dimensional heat equation with diffusivity $\nu_y, \eta_y$, respectively.
 In particular, they are stable in $H^N$ and decay as time tends to infinity.

 Next consider the orthogonal complement or by linearity assume that
 $\omega_{=}(0)=0=\theta_{=}(0)$.
 Then it holds that for every $j \leq N$, there exists $C=C(N, \eta_x, \eta_y)$
 such that
\begin{align*}
  \|\theta_{\neq}(t)\|_{H^{j}} \leq C \exp(-\eta_x t - \eta_y \frac{t^3}{12}) \|\theta_{0, \neq}\|_{H^{j}}.
\end{align*}
Thus the evolution of the temperature is stable, exponentially decreasing if
$\eta_x>0$ and exhibits enhanced dissipation if $\eta_y>0$.
Furthermore, there exists $\omega_1=\omega_1(t,\omega_0,\theta_0, \nu, \eta) \in
L^{1}_{loc}(\R, H^{N})$ and $C=C(\nu_x,\nu_y,\eta_x, \eta_y, N)$ such that 
\begin{align*}
   \|\omega(t)-\omega_1(t)\|_{H^{N}} &\leq C \exp(-\max(\eta_y,\nu_y)k^2 t^3 - \max(\eta_x,\nu_x)k^2 t) (\|\omega_0\|_{H^{N}}+\|\p_x \theta_0\|_{H^{N}}).
\end{align*}
and
\begin{align*}
  \|\omega_1(t)\|_{H^N} &\leq \min\left(t, \frac{1}{\nu_x k^2}, \frac{1}{\sqrt[3]{\nu_y k^2}},\frac{1}{\eta_x k^2}, \frac{1}{\sqrt[3]{\eta_y k^2}}\right)\\
  & \quad \exp\left(-\frac{1}{12}\min(\eta_y,\nu_y)k^2 t^3 - \frac{1}{2}\min(\eta_x,\nu_x)k^2 t\right) (\|\omega_0\|_{H^N}+\|\theta_0\|_{H^{N+1}}).
\end{align*}
Thus $\omega-\omega_1$ is stable for any choice of diffusivity parameters.
The function $\omega_1$ is stable in time if at least one diffusion coefficient
is non-zero and grows linearly if all are zero.
\end{thm}

We remark that if at least one of the vertical diffusion coefficients $\nu_y,
\eta_y$ is positive then $\omega-\omega_1$ exhibits enhanced dissipation on the
time scale $\max(\eta_y, \nu_y)^{-1/3}$.
In contrast $\omega_1$ only exhibits (enhanced) dissipation if pairs of
diffusion coefficients are positive, but is uniformly bounded if at least one
coefficient is non-zero.
As we have seen in Lemma \ref{lem:trivialunstable}, in the inviscid limit
$\omega_1$ grows linearly in $t$.

\begin{proof}[Proof of Theorem \ref{thm:stabledamping}]
  We recall that the linearized Boussinesq problem \eqref{eq:linear} is given by
  \begin{align*}
    \begin{split}
  \dt \omega + y \p_x \omega &= \nu_x \p_{xx} \omega + \nu_y \p_{yy} \omega+ \p_x \theta, \\
  \dt \theta + y \p_x \theta&= \eta_x \p_{xx} \theta + \eta_y \p_{yy} \theta,\\
  (t,x,y) &\in \R \times \T \times \R.
  \end{split}
  \end{align*}
After changing to coordinates $(x+ty,y)$ moving with the flow we obtain constant
coefficient but time-dependent differential operators on the right-hand-side.
It is therefore natural to consider an equivalent formulation by means of the
Fourier transform.
  
Let $k,\xi$ denote the Fourier variables with respect to $x,y$ and
define
\begin{align*}
  f(t,k,\xi)&= \tilde{\omega}(t,k,\xi+kt), \\
  g(t,k,\xi)&= \tilde{\theta}(t,k,\xi+kt).
\end{align*}
Then the system \eqref{eq:linear} can be equivalently expressed as
\begin{align}
  \label{eq:3}
  \begin{split}
  \dt f(t,k,\xi)&=- \nu_x k^2 f(t,k, \xi) -\nu_y (\xi-kt)^2f(t,k,\xi) +ik g(t,k,\xi), \\
  \dt g(t,k,\xi)&= -\eta_xk^2 g(t,k,\xi) -\eta_y(\xi-kt)^2 g(t,k,\xi).
  \end{split}
\end{align}
Note that $f(0)=\tilde{\omega}_0$, $g(0)=\tilde{\theta}_0$.
We in particular observe that this problem decouples with respect to the spatial
frequencies $k, \xi$ and that (only for this $\alpha=0$ case) the evolution equation
for $g$ decouples from the equation for $f$.

If $k=0$ the system simplifies to
\begin{align*}
  \dt f(t,0,\xi) &= -\nu_y \xi^2 f(t,0,\xi), \\
  \dt g(t,0,\xi)&= - \eta_y \xi^2 g(t,0,\xi),
\end{align*}
which has the explicit solutions
\begin{align*}
  f(t,0,\xi)&= \exp(-\nu_y t \xi^2) \tilde{\omega}_{0}(0,\xi),\\
  g(t,0,\xi)&= \exp(-\eta_y t \xi^2) \tilde{\theta}_{0}(0,\xi).
\end{align*}
In particular, both quantities are stable in any Sobolev norm and decay at
heat flow rates if $\eta_y$ or $\nu_y$ are positive, respectively.\\

Let next $k\neq 0$ be arbitrary but fixed.
We may then explicitly compute $g(t)$ as
\begin{align*}
  g(t,k,\eta)= \exp\left(- \eta_x k^2 t -\eta_y\int_0^t (\xi- k \tau)^2 d\tau\right) \theta_0(k,\xi).
\end{align*}
In particular, we observe that if $\eta_x>0$ we obtain exponential decay.
If $\eta_y>0$ we may compute
\begin{align*}
  \int_0^t (\xi- k \tau)^2 d\tau = \frac{1}{3k} ((kt-\xi)^3 + \xi^3)= \frac{k^3 t^3 - 3 k^2t^2\xi + 3kt \xi^2}{3k}. 
\end{align*}
We note that for fixed $k$ and $t$ this is a quadratic function in $\xi$ with
positive leading coefficient $t$ and attains its minimum for
\begin{align*}
  -3 k^2 t^2 + 6 kt \xi=0 \Leftrightarrow \xi=\frac{kt}{2}. 
\end{align*}
Hence, for any $\xi$ it holds that
\begin{align}
  \label{eq:4}
  \int_0^t (\xi- k \tau)^2 d\tau \geq \frac{1}{3k}((kt/2)^3 + (kt/2)^3)= \frac{k^2t^3}{12}.
\end{align}
Thus, it follows that $g(t)$ satisfies the pointwise estimate
\begin{align*}
  |g(t,k,\eta)| \leq \exp\left(-\eta_x k^2 t - \eta_y \frac{k^2t^3}{12}\right) |\theta_0(k,\xi)|.
\end{align*}
Hence, $g$ is stable in any Sobolev norm and exhibits exponential decay if
$\eta_x>0$ and enhanced decay if $\eta_y>0$.\\

Let us next consider $f(t)$. We may express $f$ using the following integral formula
\begin{align}
  \label{eq:5}
  \begin{split}
  f(t,k,\xi)&= \exp\left(-\nu_x k^2 t -\nu_y\int_0^t (\xi-k\tau)^2 d\tau\right) \omega_0(k,\xi) \\ & \quad + \int_0^t \exp\left(-\nu_x k^2 (t-s) -\nu_y\int_s^t (\xi-k\tau)^2 d\tau\right) \\
            & \quad \quad \times\exp\left(-\eta_x k^2 s-\eta_y \int_0^s(\xi-k\tau)^2 d\tau\right) ik\theta_0(k,\xi) ds\\
  &=: f_1(t,k,\xi)- f_2(t,k,\xi).
  \end{split}
\end{align}
The first contribution
\begin{align*}
  f_1(t,k,\xi)=\exp\left(-\nu_x k^2 t -\nu_y\int_0^t (\xi-k\tau)^2 d\tau\right) \omega_0(k,\xi)
\end{align*}
is again stable for any choice of $\nu_x,\nu_y\geq 0$ and exhibits (enhanced)
dissipation if either coefficient is positive.
Let us thus focus on the second contribution $f_2$.
Again estimating
\begin{align*}
  \int_0^s(\xi-k\tau)^2 d\tau \geq \frac{k^2s^3}{12}
\end{align*}
from below we readily see that the integral
\begin{align}
  \label{eq:6}
  \begin{split}
 & \int_0^t  \exp\left(-\nu_x k^2 (t-s) -\nu_y\int_s^t (\xi-k\tau)^2 d\tau\right) \\
  & \quad \quad \times\exp\left(-\eta_x k^2 s-\eta_y \int_0^s(\xi-k\tau)^2 d\tau\right) ds
  \end{split}
\end{align}
is bounded by a universal constant times
\begin{align*}
  \min\left(t, \frac{1}{\nu_x k^2}, \frac{1}{\sqrt[3]{\nu_y k^2}},\frac{1}{\eta_x k^2}, \frac{1}{\sqrt[3]{\eta_y k^2}}\right).
\end{align*}
In particular, if any coefficient is positive this integral is bounded.
However, if several diffusion coefficients are zero, then this integral need
not converge to zero as time tends to infinity. Thus, in order to obtain uniform
decay estimates we separately account for asymptotic behavior in terms of a
function $\omega_1$.\\

\underline{Defining $\omega_1$:}
In order to introduce ideas, let us first consider a special case.
If $\nu_x=\nu_y=0$, we observe that as $t\rightarrow \infty$ 
\begin{align*}
\int_0^t \exp\left(-\eta_x k^2 s-\eta_y \int_0^s(\xi-k\tau)^2 d\tau\right) ds \ ik\theta_0(k,\xi) 
\end{align*}
converges to a nontrivial limit.
We thus define $\omega_1$ as this limit, which for this case has the following
explicit formula:
\begin{align*}
  \omega_1:= \int_0^\infty  \exp\left(-\eta_x k^2 s-\eta_y \int_0^s(\xi-k\tau)^2 d\tau\right) ds \ ik\theta_0(k,\xi).
\end{align*}
We then observe that the difference 
\begin{align*}
  \omega_1-f_2(t)=\int_t^\infty  \exp\left(-\eta_x k^2 s-\eta_y \int_0^s(\xi-k\tau)^2 d\tau\right) ds \ ik\theta_0(k,\xi).
\end{align*}
is bounded by
\begin{align*}
  \min\left(\frac{1}{\eta_x}k^2 \exp(-k^2t), \frac{1}{\sqrt[3]{k^2 \eta_y}}\exp(-\nu_y k^2 t^3/8)\right)|k \theta_0(k,\xi)|
\end{align*}
and hence exhibits (enhanced) decay.\\

More generally, we define $\omega_1(t)$ to capture the slowest decay (in the
above example that is no decay since $\nu_x=\nu_y=0$).
We therefore split
\begin{align}
  \label{eq:7}
  \begin{split}
  & \quad \exp\left(-\nu_xk^2(t-s)- \eta_x k^2 s\right)\\
  &=  \exp(-\eta_x k^2 t) \exp(-(\nu_x-\eta_x)k^2(t-s))\\
  &= \exp(-\nu_xk^2 t) \exp(-(\eta_x-\nu_x)k^2s),
  \end{split}
\end{align}
depending on which of $\eta_x,\nu_x$ is smaller.
Similarly, we split
\begin{align}
  \label{eq:8}
  \begin{split}
 & \quad \exp\left(-\nu_y\int_s^t (\xi-k\tau)^2 - \eta_y \int_0^s(\xi-k\tau)^2\right) \\
  &= \exp\left(- \eta_y\int_0^t (\xi-k\tau)^2\right) \exp\left(-(\nu_y-\eta_y)\int_s^t (\xi-k\tau)^2\right)  \\
  &= \exp\left(- \nu_y\int_0^t (\xi-k\tau)^2\right) \exp\left(-(\eta_y-\nu_y) \int_0^s(\xi-k\tau)^2\right).
  \end{split}
\end{align}
Here the first factor is independent of $s$. Hence, for instance for $\nu_x\leq
\eta_x$, $\nu_y\leq \eta_y$ we may write
\begin{align}
  \label{eq:9}
  \begin{split}
  f_2(t)&= \exp(-\nu_xk^2 t) \exp\left(-\nu_y\int_0^t (\xi-k\tau)^2\right)  \\
  & \quad \times \int_0^t \exp(-(\eta_x-\nu_x)k^2s)
  \exp\left(-(\eta_y-\nu_y) \int_0^s(\xi-k\tau)^2)\right) ds \ ik \tilde{\theta}_0(k,\xi).
  \end{split}
\end{align}
If both $\eta_x=\nu_x$ and $\eta_y=\nu_y$ the inner integral simplifies to $t$
and $f_2(t)$ decays exponentially. In this case we simply set
$\omega_{1}(t):=f_2(t)$.
In the following we restrict to
the case where at least one pair is not equal.

Then the inner integral is uniformly bounded by a
uniform constant times
\begin{align*}
  \min(\frac{1}{k^2|\eta_x-\nu_x|}, \frac{1}{\sqrt[3]{k^2|\eta_y-\nu_y|}}).
\end{align*}
We therefore aim to define $\omega_1(t)$ by passing to the limit $t\rightarrow \infty$ in
the inner integral of equation \eqref{eq:9}.
We distinguish the following four cases:

If $\nu_x\leq \eta_x$, $\nu_y\leq \eta_y$ we define
\begin{align*}
  \omega_1(t) &:= \exp(-\nu_xk^2 t) \exp(- \nu_y\int_0^t (\xi-k\tau)^2)  \\
  & \quad \times \int_0^\infty \exp(-(\eta_x-\nu_x)k^2s)
  \exp\left(-(\eta_y-\nu_y) \int_0^s(\xi-k\tau)^2) ds\right) \ ik \tilde{\theta}_0(k,\xi).
\end{align*}
and observe that
\begin{align*}
  |f_2(t)- \omega_1(t)| &= \exp(-\nu_xk^2 t) \exp(- \nu_y\int_0^t (\xi-k\tau)^2)  \\
  &\quad \times \int_t^\infty \exp(-(\eta_x-\nu_x)k^2s)
  \exp\left(-(\eta_y-\nu_y) \int_0^s(\xi-k\tau)^2) ds\right) \ |k \tilde{\theta}_0(k,\xi)| \\
                        &\leq   C \exp(-\nu_xk^2 t) \exp(- \nu_y\int_0^t (\xi-k\tau)^2) \\
  & \quad \exp\left(-(\eta_x-\nu_x)k^2 t - (\eta_y-\nu_y)k^2 t^3/12\right) |k \tilde{\theta}_0(k,\xi)|
\end{align*}
exhibits (enhanced) dissipation with the larger of the coefficients.

If $\nu_x\geq \eta_x$, $\nu_y\geq \eta_y$ we similarly define
\begin{align*}
  \omega_1(t) &:= \exp(-\eta_x k^2 t) \exp(- \eta_y\int_0^t (\xi-k\tau)^2)  \\
              & \quad \times \int_0^\infty \exp(-(\eta_x-\nu_x)k^2\sigma)
  \exp\left(-(\eta_y-\nu_y) \int_{t-\sigma}^t(\xi-k\tau)^2\right) d\sigma \ ik \tilde{\theta}_0(k,\xi),
\end{align*}
where we introduced the change of variables $s=t-\sigma$ and extended the domain
of integration from $\sigma \in [0,t]$ to $\sigma \in [0,\infty)$.
We remark that here the inner integral still depends on $t$.
By an analogous calculation we then again observe that $f_2(t)-\omega_1(t)$
exhibits (enhanced) dissipation with the larger of the coefficients.

Finally, for $\eta_x\leq \nu_x$ and $\eta_y\geq \nu_y$ we define
\begin{align*}
  \omega_1(t)&:= \exp(-\eta_x k^2 t) \exp\left(- \nu_y \int_0^t (\xi-k\tau)^2\right)\\
  & \quad \times \int_{-\infty}^\infty \exp\left(-(\eta_x-\nu_x)k^2\min(s,0)) -(\eta_y-\nu_y) \int_{\min(s,t)}^t(\xi-k\tau)^2\right) ds \ ik \tilde{\theta}_0(k,\xi),
\end{align*}
and analogously for $\eta_x\geq \nu_x$, $\eta_y\leq \nu_y$.
\end{proof}

\subsection{The Effects of Hydrostatic Balance}
\label{sec:alpha}
In the previous Section \ref{sec:homogen} we have shown that in the case $\alpha=0$
very weak partial dissipation (just one non-zero coefficient) is sufficient to
obtain asymptotic stability and decay rates and that the vorticity can be
decomposed into a slower (or not all) decaying part $\omega_1$ and a fast
decaying part $\omega-\omega_1$.
For that setting we could exploit that the equation for $\theta$ decouples and
that we can thus first solve for $\theta(t)$ and subsequently for $\omega(t)$.

If $\alpha> 0$ this decoupling structure is lost and we obtain the following
system at each Fourier frequency:
\begin{align}
  \label{eq:10}
  \dt
  \begin{pmatrix}
    \tilde{\omega} \\ \tilde{\theta}
  \end{pmatrix}
  =
  \begin{pmatrix}
    -\nu_xk^2 - \nu_y(\xi-\beta kt)^2 & ik \\
    \frac{ik \alpha}{k^2+(\xi- \beta kt)^2} & - \eta_x k^2 - \eta_y (\xi-\beta kt)^2
  \end{pmatrix}
           \begin{pmatrix}
    \tilde{\omega} \\ \tilde{\theta}
  \end{pmatrix}.
\end{align}
We note that if $\beta\neq 0$ all coefficients except $ik$ are time-dependent, which makes
this problem very challenging.
As a first step we hence discuss the setting without shear, $\beta=0$, and
introduce two methods of proof. The first method is an adaptation of energy
methods commonly used in the nonlinear problem and second, more precise result
constructs explicit solutions.

\subsubsection{The Case without Shear}
\label{sec:woshear}

In this section we consider the linearized problem around $(\omega,
\theta)=(0,\alpha y)$ with $\alpha>0$.
The corresponding nonlinear problem has been studied in \cite{doering2018long}
for the setting of full dissipation and of vertical dissipation.
As a first method of proof in Proposition \ref{prop:woshear} we adapt energy arguments which are well-known for
the nonlinear problem (see \cite{doering2018long,li2016global,larios2013global} ) to this linear setting. 
This approach has the benefit of a very simple and robust structure.
However, it does not precisely capture the effects of the various diffusion
coefficients.
As a second method in Proposition \ref{prop:woshear2} we hence derive explicit solutions of the ODE systems in
Fourier variables. Here we crucially exploit the lack of shear and hence
time-independence of the coefficients.

\begin{prop}
  \label{prop:woshear}
  Let $\alpha>0$ and $\nu_x,\nu_y,\eta_x,\eta_y\geq 0$ be given.
  Then for any initial data $(\omega_0, \theta) \in H^{N}\times H^{N+1}$ the
  solution $\omega, \theta$ of the linearized problem
  \begin{align}
    \label{eq:11}
    \begin{split}
    \dt \omega &= \nu_x \p_x^2 \omega + \nu_y \p_y^2 \omega + \p_x \theta,\\
    \dt \theta + \alpha \p_x \Delta^{-1} \omega &= \eta_x \p_{x}^2 \theta + \eta_y \p_y^2 \theta. 
    \end{split}
  \end{align}
  is stable and satisfies
  \begin{align*}
    \frac{d}{dt} (\alpha \|\omega\|_{H^N}^2 + \|\nabla \theta\|_{H^N}^2) + \nu_x \|\p_x \omega\|_{H^N}^2 + \nu_y \|\p_y \omega\|_{H^N}^2 + \eta_x\|\p_x \theta\|_{H^N}^2 + \eta_y \|\p_y \theta\|_{H^N}^2 =0
  \end{align*}
\end{prop}

\begin{proof}
  We note that all differential operators in \eqref{eq:11} are linear and
  involve constant coefficients. Hence, the problem decouples in frequency and
  we may without loss of generality restrict to $N=0$ and studying single modes
  $(k, \xi)$.
  Here, the $x$-average decouples and evolves by heat flow, so we further
  restrict to analyzing $k\neq 0$.

  Then after a Fourier transform we obtain
  \begin{align}
    \label{eq:12}
    \begin{split}
    \dt
    \begin{pmatrix}
      \tilde{\omega}\\
      \tilde{\theta}
    \end{pmatrix}
=
    \begin{pmatrix}
      -\nu_x k^2 - \nu_y \xi^2 & ik \\
      \frac{ik \alpha}{k^2+ \xi^2} & -\eta_x k^2 - \eta_y \xi^2
    \end{pmatrix}
                                     \begin{pmatrix}
                                       \tilde{\omega}\\
                                       \tilde{\theta}
                                     \end{pmatrix}.
    \end{split}
  \end{align}
As we discuss in Proposition \ref{prop:woshear2} this constant coefficient ODE
system can be solved explicitly by means of the matrix
exponential function. However, for this proposition we instead use an
energy argument which exploits anti-symmetry:
If we multiply $\tilde{\theta}$ by $\sqrt{k^2+ \xi^2}$ and $\tilde{\omega}$ by $\sqrt{\alpha}$ our system reads
\begin{align*}
      \dt
    \begin{pmatrix}
      \sqrt{\alpha} \tilde{\omega}\\
      \sqrt{k^2+\xi^2}\tilde{\theta}
    \end{pmatrix}
=
    \begin{pmatrix}
      -\nu_x k^2 - \nu_y \xi^2 & \frac{ik \sqrt{\alpha}}{\sqrt{k^2+\xi^2}} \\
      \frac{ik \sqrt{\alpha}}{\sqrt{k^2+ \xi^2}} & -\eta_x k^2 - \eta_y \xi^2
    \end{pmatrix}
                                     \begin{pmatrix}
                                       \sqrt{\alpha}\tilde{\omega}\\
                                       \sqrt{k^2+\xi^2}\tilde{\theta}
                                     \end{pmatrix}.
\end{align*}
The off-diagonal matrix entries are the same and purely imaginary. Therefore,
they cancel when considering $M+\overline{M}^T$ and
\begin{align*}
  \frac{d}{dt} (|\sqrt{\alpha}\tilde{\omega}|^2+ |\sqrt{k^2+\xi^2}\tilde{\theta}|^2)= - (\nu_x k^2 + \nu_y \xi^2) |\sqrt{\alpha}\tilde{\omega}|^2 - (\eta_x k^2+\eta_y \xi^2) |\sqrt{k^2+\xi^2}\tilde{\theta}|^2.
\end{align*}
This energy functional is hence non-increasing and we obtain decay estimates in
terms of $\min(\nu_x, \eta_x)k^2$ and $\min(\nu_y,\eta_y)\xi^2$.
 We further remark that after multiplying by $\frac{1}{\sqrt{k^2+\xi^2}}$, this
 is equivalent to an estimate on the velocity and density
 \begin{align*}
   \alpha\|u\|_{H^N}^2 + \|\theta\|_{H^N}^2,
 \end{align*}
 see \cite{doering2018long} for a nonlinear analogous estimate.
\end{proof}

As an alternative, more fragile but also more precise approach, we may compute
explicit solution operators in Fourier variables.
\begin{prop}
  \label{prop:woshear2}
  Let $\alpha>0$ and $\nu_x, \nu_y, \eta_x, \eta_y\geq 0$ be given.
  Then for any initial data $(\omega_0, \theta_0) \in H^N \times H^{N+1}$ of the
  linearized problem \eqref{eq:11} is stable. Furthermore, for every frequency
  $(k, \xi)$ and
  \begin{align*}
    \alpha \neq \frac{k^2+\xi^2}{k^2} \left(\frac{\eta_x-\nu_x}{2}k^2 +\frac{\eta_y-\nu_y}{2}\xi^2\right)^2=: \alpha^*,
  \end{align*}
  there exists a basis $(v_1,v_2)$ and constants
  \begin{align*}
    \lambda_{1,2}
= -\frac{\eta_x+\nu_x}{2}k^2 - \frac{\eta_y+\nu_y}{2}\xi^2 
\pm \sqrt{\left(\frac{\eta_x-\nu_x}{2}k^2 +\frac{\eta_y-\nu_y}{2}\xi^2\right)^2 - \alpha \frac{k^2}{k^2+\xi^2}},
  \end{align*}
  such that in this basis the evolution of $(\tilde{\omega}, \tilde{\theta})$ is
  given by
  \begin{align*}
    \begin{pmatrix}
      e^{\lambda_1 t} \\ e^{\lambda_2 t}
    \end{pmatrix}.
  \end{align*}
  We in particular observe that for all $\alpha \in (0,\alpha^*)$ it holds that
  \begin{align*}
    \lambda_1, \lambda_2 <0
  \end{align*}
  and for $\alpha>\alpha^*$
  \begin{align*}
    \text{Re}(\lambda_1)=\text{Re}(\lambda_2)=  -\frac{\eta_x+\nu_x}{2}k^2 - \frac{\eta_y+\nu_y}{2}\xi^2.
  \end{align*}
\end{prop}
\begin{proof}[Proof of Proposition \ref{prop:woshear2}]
  We recall that equation \eqref{eq:11} is equivalent to the the ODE system \eqref{eq:12}
  \begin{align*}
    \begin{split}
    \dt
    \begin{pmatrix}
      \tilde{\omega}\\
      \tilde{\theta}
    \end{pmatrix}
=
    \begin{pmatrix}
      -\nu_x k^2 - \nu_y \xi^2 & ik \\
      \frac{ik \alpha}{k^2+ \xi^2} & -\eta_x k^2 - \eta_y \xi^2
    \end{pmatrix}
                                     \begin{pmatrix}
                                       \tilde{\omega}\\
                                       \tilde{\theta}
                                     \end{pmatrix},
    \end{split}
  \end{align*}
  at each frequency $(k,\xi)$.
We denote the coefficient matrix as
\begin{align*}
  M= \begin{pmatrix}
      -\nu_x k^2 - \nu_y \xi^2 & ik \\
      \frac{ik \alpha}{k^2+ \xi^2} & -\eta_x k^2 - \eta_y \xi^2
    \end{pmatrix}. 
\end{align*}
  Since $M$ is time-independent, we obtain a solution
  in terms of the matrix exponential function:
  \begin{align*}
    \begin{pmatrix}
      \tilde{\omega}\\
      \tilde{\theta}
    \end{pmatrix}
= \exp\left( t M \right)        \begin{pmatrix}
                                       \tilde{\omega}_0\\
                                       \tilde{\theta}_0
                                     \end{pmatrix}.
  \end{align*}
It thus remains to explicitly compute the matrix exponential $\exp( t M)$.
We recall that the eigenvalues of a $2 \times 2$ matrix are given by the roots
of the characteristic polynomial
\begin{align*}
  \lambda^2 - \text{tr}(M) \lambda + \text{det}(M).
\end{align*}
We thus obtain
\begin{align*}
  \lambda_{1,2}&= \frac{1}{2} \text{tr}(M) \pm \sqrt{(\frac{\text{tr}(M)}{2})^2-\text{det}(M)}\\
&= -\frac{\eta_x+\nu_x}{2}k^2 - \frac{\eta_y+\nu_y}{2}\xi^2 
\pm \sqrt{\left(\frac{\eta_x-\nu_x}{2}k^2 +\frac{\eta_y-\nu_y}{2}\xi^2\right)^2 - \alpha \frac{k^2}{k^2+\xi^2}},
\end{align*}
where we used that $(\frac{a+d}{2})^2-ad +bc =(\frac{a-d}{2})^2 +bc$.
For simplicity of notation let us denote
\begin{align*}
  r= \sqrt{\left(\frac{\eta_x-\nu_x}{2}k^2 +\frac{\eta_y-\nu_y}{2}\xi^2\right)^2 - \alpha \frac{k^2}{k^2+\xi^2}}
\end{align*}
Then corresponding eigenvectors are given by
\begin{align*}
  \begin{pmatrix}
    \frac{k^2+\xi^2}{k\alpha} \left( i \left(\frac{\eta_x-\nu_x}{2}k^2 + \frac{\eta_y-\nu_y}{2}\xi^2\right) \pm r \right)\\ 1
  \end{pmatrix}
\end{align*}
and
\begin{align*}
  \exp(t M) =
  \begin{pmatrix}
    v_1 & v_2
  \end{pmatrix}
          \begin{pmatrix}
            \exp(t \lambda _1) & 0 \\
            0 & \exp( t \lambda_2)
          \end{pmatrix}
                \begin{pmatrix}
                  v_1 & v_2
                \end{pmatrix}^{-1},
\end{align*}
if $\lambda_1 \neq \lambda_2$.
This slightly degenerates if $\alpha=\alpha^*$ (if $r=0$) with a cyclic subspace and growth with a
factor $t$. We omit this case for brevity.

It remains to discuss the size of the eigenvalues.
We note that if 
\begin{align*}
  \alpha > \frac{k^2+\xi^2}{k^2} \left(\frac{\eta_x-\nu_x}{2}k^2 +\frac{\eta_y-\nu_y}{2}\xi^2\right)^2=: \alpha^*,
\end{align*}
then $r$ is strictly imaginary and 
\begin{align}
  \text{Re}(\lambda_1)=\text{Re}(\lambda_2)=  -\frac{\eta_x+\nu_x}{2}k^2 - \frac{\eta_y+\nu_y}{2}\xi^2.
\end{align}
In particular,
\begin{align*}
  |e^{ \lambda_1 t}| = |e^{\lambda_2 t}|= \exp\left(-t \left(  \frac{|\eta_x-\nu_x|}{2}k^2 +\frac{|\eta_y-\nu_y|}{2}\xi^2\right)\right)
\end{align*}
both decay exponentially even if only some of the dissipation coefficients are
non-zero.

Moreover, if $0< \alpha< \alpha^*$ the eigenvalues $\lambda_1, \lambda_2$ are
distinct and real-valued and
\begin{align*}
  \sqrt{\left(\frac{\eta_x-\nu_x}{2}k^2 +\frac{\eta_y-\nu_y}{2}\xi^2\right)^2 - \alpha \frac{k^2}{k^2+\xi^2}} < \left| \frac{\eta_x-\nu_x}{2}k^2 +\frac{\eta_y-\nu_y}{2}\xi^2\right|\\
  \leq \frac{|\eta_x-\nu_x|}{2}k^2 +\frac{|\eta_y-\nu_y|}{2}\xi^2.
\end{align*}
Therefore
\begin{align*}
  \lambda_2 > \frac{\min(\eta_x,\nu_x)}{2}k^2 +\frac{\min(\eta_y,\nu_y)}{2}\xi^2
\end{align*}
is positive even if multiple dissipation coefficients are zero.
\end{proof}
This explicit solution shows that the dependence of sharp decay rates on the
parameters is more subtle than captured in Proposition \ref{prop:woshear}.
However, in the case $\beta\neq 0$ of the following section explicit solutions
become infeasible to compute and we hence rely on more robust but less precise
energy arguments.

\subsubsection{On the Interaction of Shear and Hydrostatic balance}
\label{sec:vert}

In this section we consider the linearized problem with $\alpha>0$ and $\beta=1$ and with partial dissipation:
\begin{align}
  \label{eq:13}
  \dt
  \begin{pmatrix}
    \tilde{\omega} \\ \tilde{\theta}
  \end{pmatrix}
  =
  \begin{pmatrix}
    -\nu_xk^2 - \nu_y(\xi- kt)^2 & ik \\
    \frac{ik \alpha}{k^2+(\xi- kt)^2} & - \eta_x k^2 - \eta_y (\xi- kt)^2
  \end{pmatrix}
           \begin{pmatrix}
    \tilde{\omega} \\ \tilde{\theta}
  \end{pmatrix}.
\end{align}
Unlike the setting studied in Section \ref{sec:homogen} here the evolution of
$\tilde{\theta}$ does not decouple anymore and most coefficients are
time-dependent.
Therefore, this problem cannot be solved explicitly
by means of a matrix exponential function and also does not easily decouple into
second order equations as in Section \ref{sec:inviscid}.

Instead, we aim at adapting the energy method discussed in Proposition
\ref{prop:woshear} of Section
\ref{sec:homogen} to this setting.
\begin{prop}
  \label{prop:wshear}
  Let $\omega, \theta$ be a solution of the problem \eqref{eq:13}.
  Then it holds that
  \begin{align*}
    & \quad \alpha \|\omega(t)\|_{H^N}^2 + \|(\p_x, \p_y-t\p_x)\theta(t)\|_{H^N}^2 \\
    & \leq C (1+t^2) \exp(-\min(\nu_x,\eta_x)t 
    -\min(\nu_y,\eta_y)t^{3}/12) (\alpha\|\omega_0\|_{H^N} + \|\nabla \theta_0\|_{H^N}).
  \end{align*}
   In particular, if at least one of $\min(\nu_x, \eta_x)$ or $\min(\nu_y,\eta_y)$ is
  positive (that is, pairs of entries are non-zero), the system is
  asymptotically stable.

 Furthermore, it holds that
  \begin{align*}
    \alpha \|v\|_{H^N}^2 + \|\theta\|_{H^N}^2 \leq C \exp(-\min(\nu_x,\eta_x)t -\min(\nu_y,\eta_y)t^{3}/8) (\alpha\|\omega_0\|_{H^{N+1}} + \|\nabla \theta_0\|_{H^N})
  \end{align*}
  Thus we may trade higher regularity of $\omega_0$ for a uniform bound on the
  velocity. 
\end{prop}

\begin{proof}
We recall that the problem under consideration is given by the following
time-dependent system of ODEs:
\begin{align*}
  \dt
  \begin{pmatrix}
    \tilde{\omega} \\ \tilde{\theta}
  \end{pmatrix}
  =
  \begin{pmatrix}
    -\nu_x k^2 - \nu_y(\xi-kt)^2 & ik \\
    \frac{ik \alpha}{k^2+(\xi-kt)^2} & -\eta_x k^2 - \eta_y(\xi-kt)^2
  \end{pmatrix}
 \begin{pmatrix}
    \tilde{\omega} \\ \tilde{\theta}
  \end{pmatrix}.
\end{align*}
As the coefficient matrix does not exhibit anti-symmetry in this formulation,
we aim to use a change of basis similar to the one of Section \ref{sec:woshear}.
That is, we consider
\begin{align*}
  \begin{pmatrix}
    \sqrt{\alpha} \tilde{\omega} \\ \sqrt{k^2+(\xi-kt)^2}\tilde{\theta}
  \end{pmatrix}.
\end{align*}
Here we obtain an additional correction term involving 
\begin{align*}
  \dt \sqrt{k^2+(\xi-kt)^2} = \frac{k(kt-\xi)}{\sqrt{k^2+(\xi-kt)^2}}.
\end{align*}
Inserting this ansatz into the equation \eqref{eq:13} we obtain the following system:
\begin{align*}
   \dt
  \begin{pmatrix}
    \sqrt{\alpha} \tilde{\omega} \\ \sqrt{k^2+(\xi-kt)^2}\tilde{\theta}
  \end{pmatrix}
  =
  \begin{pmatrix}
    -\nu_x k^2 - \nu_y(\xi-kt)^2 & \frac{ik\sqrt{\alpha}}{\sqrt{k^2+(\xi-kt)^2}} \\
   \frac{ik\sqrt{\alpha}}{\sqrt{k^2+(\xi-kt)^2}}  & -\eta_x k^2 - \eta_y(\xi-kt)^2 + \frac{t-\frac{\xi}{k}}{1+(t-\frac{\xi}{k})^2}
 \end{pmatrix}\\
 \begin{pmatrix}
    \sqrt{\alpha} \tilde{\omega} \\ \sqrt{k^2+(\xi-kt)^2}\tilde{\theta}
  \end{pmatrix}.
\end{align*}
As the off-diagonal entries are identical and purely imaginary, we deduce that
\begin{align*}
  \dt (\alpha |\tilde{\omega}|^2 + (k^2+(\xi-kt)^2) |\tilde{\theta}|^2)
  &= (-\nu_x k^2 - \nu_y(\xi-kt)^2) \alpha |\tilde{\omega}|^2 \\
  & \quad + \left(-\eta_x k^2 - \eta_y(\xi-kt)^2 + \frac{t-\frac{\xi}{k}}{1+(t-\frac{\xi}{k})^2}\right) (k^2+(\xi-kt)^2) |\tilde{\theta}|^2.
\end{align*}
The right-hand-side thus contains terms yielding exponential decay due to
dissipation (see Section \ref{sec:homogen}) as well as possible algebraic growth due to
\begin{align*}
  \exp \left(\int_0^T \max(0,\frac{t-\frac{\xi}{k}}{1+(t-\frac{\xi}{k})^2}) dt \right)\\
  = \exp \left(\int_0^T 1_{t>\frac{\xi}{k}} \frac{d}{dt} \frac{1}{2} \ln(1+(t-\frac{\xi}{k})^2)\right)
  = 1_{T>\frac{\xi}{k}>0} \sqrt{1+(T-\frac{\xi}{k})^2} \leq \sqrt{1+T^2}.
\end{align*}
We remark that here we could pass to the positive part since negative
contributions are beneficial in  energy estimates.
Combining both bounds we obtain the desired result.\\

We may repeat the same argument for
\begin{align*}
  \begin{pmatrix}
  \frac{\sqrt{\alpha}}{\sqrt{k^2+(\xi-kt)^2}} \tilde{\omega} \\ \tilde{\theta}.
  \end{pmatrix}
\end{align*}
However, here
\begin{align*}
  \dt \frac{1}{\sqrt{k^2+(\xi-kt)^2}} = \frac{(\xi-kt)k}{(\sqrt{k^2+(\xi-kt)^2})^3}
\end{align*}
has an opposite sign (it grows until $t=\frac{\xi}{k}$ and decreases
afterwards).
In particular,
\begin{align*}
  \int_0^\infty \max\left(0,\frac{\frac{\xi}{k}-t}{1+(\frac{\xi}{k}-t)^2}\right) dt \leq C \left(1+\left(\frac{\xi}{k}\right)^2\right)
\end{align*}
corresponds to a loss of two derivatives compared to
$\frac{1}{\sqrt{k^2+\xi^2}}\omega_0$ and thus one derivative of $\omega_0$.
\end{proof}

We remark that due to the less explicit structure of the solutions, these
results are less optimal than those of previous sections. However, they serve to
highlight how the interaction of shear and hydrostatic balance introduces a
stronger coupling between the vorticity and temperature.

\section{The Nonlinear Full Dissipation Case}
\label{sec:fulldissipation}
In this section we consider the nonlinear, viscous Boussinesq problem.
We remark that questions of well-posedness or asymptotic stability for partial
dissipation problems here are very challenging and for instance considered in
\cite{doering2018long,li2016global,larios2013global} or \cite{chae2006global,chae1999local}.
For this reason we instead consider the full dissipation case and aim to obtain a
more precise description of asymptotic behavior near the stationary solutions
\begin{align}
  \label{eq:14}
\omega=\beta, \ v=(\beta y, 0), \ \theta = \alpha y,
\end{align}
which combine both shear flow and hydrostatic balance.

Here we consider two distinct cases. If $\alpha\geq 0$ is ``small'', we adapt the methods developed by Bedrossian, Vicol and Wang
\cite{bedrossian2016sobolev} for the $2D$ Navier-Stokes equations near Couette
flow to the Boussinesq setting (see Section \ref{sec:smallalpha}). See also the recent work by Luo \cite{luosobolev}, who adapt
these methods to the hyperviscosity equations near Couette flow.
As a second case we consider the setting where $\alpha>0$ is ``large'' and
combined with a shear $\beta=1$ (see Section \ref{sec:largealpha}). There we
combine classical energy argument approaches for perturbations of hydrostatic
balance (e.g. see \cite{doering2018long}) with the bootstrap approach of \cite{bedrossian2016sobolev}.  

We recall that the full nonlinear Boussinesq equations with (isotropic)
viscosity $\nu>0$  and thermal diffusivity $\eta>0$ are given by
\begin{align*}
  \dt \omega + v \cdot \nabla \omega &= \nu \Delta \omega + \p_x \theta, \\
  \dt \theta + v \cdot \nabla \theta &= \eta \Delta \theta, \\
  v &= \nabla^{\bot} \Delta^{-1} \omega.
\end{align*}
Given a stationary solution of the form \eqref{eq:14} we consider the equation
for perturbations $\omega= \beta+ \omega^*, v =
(\beta y,0)+v^*$, $\theta =\alpha y + \theta^*$:
\begin{align*}
  \dt \omega^* + \beta y \p_x \omega^* + v^* \cdot \nabla \omega^*&= \nu \Delta \omega^* + \p_x \theta^*, \\
  \dt \theta^* + \beta y \p_x \theta^* + v^* \cdot \nabla \theta^* - \alpha v^*_{2} &= \eta \Delta \theta^*.
\end{align*}
We view this problem as a modification of the transport equation $\dt + \beta y \p_x$
and with slight abuse of notation reuse $\omega, v, \theta$ to denote
\begin{align*}
  \omega(t,x,y)&:= \omega^*(t,x+t \beta y,y), \\
  v(t,x,y)&:= v^*(t,x+t \beta y,y), \\
  \theta(t,x,y)&:= \omega^*(t,x+t\beta y,y),
\end{align*}
and define
\begin{align*}
 \nabla_t=(\p_x, \p_y-t\beta \p_x), \ \Delta_t=(\p_x^2+(\p_y-t\beta \p_x^2)).
\end{align*}
With these conventions the nonlinear Boussinesq equations read:
\begin{align}
  \label{eq:B}
  \begin{split}
  \dt \omega + \nabla_t \Delta_t^{-1} \omega \cdot \nabla_t \omega &= \nu \Delta_t \omega + \p_x \theta, \\
  \dt \theta + \nabla_t \Delta^{-1} \omega \cdot \nabla_t \theta  &= \eta \Delta_t \theta + \alpha \p_x \Delta_{t}^{-1} \omega.
  \end{split}
\end{align}
We then aim to show that for sufficiently small initial data this system of
equations is asymptotically stable and $\omega, \theta$ converge to zero as
$t\rightarrow \infty$ at enhanced dissipation rates.
Here we first consider the question of stability in the setting where $\alpha$ is ``small'' in Section \ref{sec:smallalpha}. Subsequently we discuss the setting of
$\alpha>1$ in Section \ref{sec:largealpha}.
Finally, in Section \ref{sec:bounds} we explain how the stability results we
obtained can be used to derive enhanced dissipation rates.

\subsection{Shear and Small Hydrostatic Balance}
\label{sec:smallalpha}
In this section we consider the nonlinear asymptotic stability for the case when
$\alpha$ is ``small''. In this case a quantity such as $\|\omega\|_{H^N}^2 +
\|\nabla \theta\|_{H^N}^2$ considered in Section \ref{sec:alpha} is of limited use.
Instead we aim to exploit shearing behavior for $\beta\neq 0$ following the bootstrap/multiplier approach employed in
\cite[Section 2]{bedrossian2016sobolev} for the Navier-Stokes problem near
Couette flow with relatively minor changes. The case of ``large'' $\alpha$ is
considered in Theorem \ref{thm:largealpha}.
For simplicity of notation we in the following consider the case $\beta=1$.
\begin{thm}
  \label{thm:nonlinear}
  Let $N\geq 5$, $\beta=1$ and let $\epsilon_1\ll \min(\nu,\eta)^{1/2}$, $\epsilon_{2} \ll
  \sqrt{\eta}\sqrt{\nu} \epsilon_1$ and suppose that $0\leq \alpha < \eta^{1/2}\nu^{1/3}\frac{\epsilon_2}{\epsilon_1}$.
  Then if $\|\omega_0\|_{H^N} \leq \epsilon_1$ and $\|\theta_0\|_{H^N}\leq
  \epsilon_2$, the unique global solution with this initial data satisfies
  \begin{align}
    \label{eq:27}
    \begin{split}
    \|\omega\|_{L^\infty_t H^N}^2 + \nu \|\nabla_t \omega\|_{L^2 H^N}^2 + \|\nabla_t \Delta_t^{-1} \omega\|_{L^2 H^N}^2  &\leq 8 \epsilon_1^2, \\
    \|\theta\|_{L^\infty_t H^N}^2 + \eta \|\nabla_t \theta\|_{L^2 H^N}^2 &\leq 8 \epsilon^2.
    \end{split}
  \end{align}
\end{thm}

\begin{rem}
  \begin{itemize} 
  \item Here we study the regime of ``small'' $\alpha$, where stabilizing by mixing
    is dominant. In contrast, if $\alpha$ is ``large'' we may make use of (higher
    regularity analogues) of conserved quantities, which we discuss in Section \ref{sec:largealpha}. 
  \item 
  A constraint of the form $\epsilon_1 \ll \nu^{1/2}$ is also imposed in
  \cite{bedrossian2016sobolev}. In view of instabilities in the inviscid setting
  some constraint of this type is likely necessary, though weaker asymptotic
  stability results may persist under weaker constraints \cite{dengZ2019,dengmasmoudi2018}.
  The constraints imposed on $\epsilon_1, \alpha$ and, in particular, $\epsilon_2$ are
  probably quite far from optimal but allow for a relatively simple proof.
  See Proposition \ref{prop:estimates} for details.
\item 
  We remark that global well-posedness results, also for larger data and partial
  dissipation, have been already previously obtained in several works by other
  methods, for example \cite{doering2018long,li2016global,larios2013global,chae1999local,chae2006global}. This method's focus instead lies on establishing
  stability of the two parameter family, as well as damping and convergence rates (see Section \ref{sec:bounds}).
  The convergence rates are derived in Section \ref{sec:bounds} as a corollary of this theorem's
  bounds.
\item Instead of bounds relating $\epsilon_1, \epsilon_2$ and $\alpha$, we could
  for instance denote $\epsilon_1=\epsilon$ and require $\epsilon_2=\epsilon^2$,
  $\alpha< \epsilon$.
\item In view of the existing well-posedness results we just referenced and the
  linear results of Section \ref{sec:vert} the constraint on $\alpha$ here is
  probably far from optimal. However, it allows us to treat the effects of
  hydrostatic balance perturbatively. 
\end{itemize}
\end{rem}

We make use of the following multiplier constructed in
\cite{bedrossian2016sobolev}.

\begin{lem}[\cite{bedrossian2016sobolev}]
  There exists a Fourier multiplier $M(t,k,\xi)$ with the following properties:
  \begin{align}
    M(0,k,\xi)&= M(t,0,\xi)=1, \\
    1&\geq M(t,k,\xi)\geq c, \\
    -\frac{\dot{M}}{M}&\geq \frac{|k|}{k^2+|\xi-kt|^2} \text {for } k \neq 0, \\
    \left| \frac{\p_\xi M(k,\xi)}{M(k,\xi)}\right| &\lesssim \frac{1}{|k|} \text{ for } k \neq 0\text{, uniformly in }\xi, \\
  \label{eq:vbound}  1 &\lesssim \nu^{-1/6}(\sqrt{-\dot{M}M(t,k,\xi)}+ \nu^{1/2}|k,\xi-kt|), \\
    \sqrt{-\dot{M}M(t,k,\eta)}&\lesssim \langle \eta-\xi \rangle \sqrt{-\dot{M}M(t,k,\xi)}.
  \end{align}
\end{lem}

For later reference we also recall from \cite{bedrossian2016sobolev} that
\eqref{eq:vbound} implies that for any function $f$ it holds that
\begin{align}
  \label{eq:17}
  \|f_{\neq} \|_{L^2H^N}\lesssim \nu^{-1/6}(\nu^{1/2}\|\nabla_t A f_{\neq}\|_{L^2 L^2}+ \|\sqrt{-\dot M M}\langle D \rangle^N f_{\neq} \|_{L^2 L^2}).
\end{align}

Given $M$ we are ready to define the main quantities of our proof:
\begin{defi}
  Let $N \in \N, N \geq 5$ be given and define the Fourier multiplier
  \begin{align}
  A=  M \langle D \rangle^N,
  \end{align}
  where $D=i \nabla$ is the Fourier multiplier $(k,\xi)$.

  We then define two energies:
 \begin{align}
  E_{\omega}(T)&:= \|A\omega\|_{L^\infty L^2((0,T))}^2 + \nu \|\nabla_t A \omega\|_{L^2 L^2((0,T))} + \|\sqrt{-\dot M M} \langle D \rangle^N \omega\|_{L^2 L^2((0,T))}^2, \\
 E_{\theta}(T)&:= \|A\theta\|_{L^\infty L^2((0,T))}^2 + \eta \|\nabla_t A \theta\|_{L^2 L^2((0,T))} + \|\sqrt{-\dot M M} \langle D \rangle^N \theta\|_{L^2((0,T))}^2. 
\end{align} 
\end{defi}
As $M$ is comparable to $1$, we in the following may replace \eqref{eq:27} by the estimates
\begin{align}
  \label{eq:15}
  E_{\omega}(T)\leq 8 \epsilon_1^2, \\
  \label{eq:16}
  E_{\theta}(T)\leq 8 \epsilon_2^2.
\end{align}

We then follow a classic bootstrap approach (e.g. see \cite{bedrossian2013inviscid}):
\begin{itemize}
\item By local well-posedness there exists at least some small time $T>0$ for which
  \eqref{eq:15} and \eqref{eq:16} hold. This is established in Proposition \ref{prop:initialbootstrap}. 
\item Since these are closed conditions, there exists some maximal time $T_*$
  for which \eqref{eq:15} and \eqref{eq:16} hold. Suppose for the sake of
  contradiction that $T_*<\infty$.
  Then we show in Proposition \ref{prop:improve} that on $(0,T_*)$  improved estimates with $4\epsilon_1^2$ and $4\epsilon_2^2$ hold. 
\item Therefore, by local continuity \eqref{eq:15} and \eqref{eq:16} remain true
  for an additional small time. Hence, $T_*<\infty$ was not maximal, which
  contradicts the assumption of the previous step.
  Therefore the maximal time has to have been infinity, which concludes the proof. 
\end{itemize}

\begin{prop}
  \label{prop:initialbootstrap}
  Let $0<\epsilon_1< \min(\nu,\eta)^{1/2}$ and $0<\epsilon_2< \sqrt{\nu \eta} \epsilon_1$ and
  $N \in \N, N \geq 5$.
  
  Suppose that the initial data $\omega_0, \theta_0$ satisfies
  \begin{align*}
    \|\omega_0\|_{H^N}^2 &\leq \frac{1}{10} \epsilon_1^2, \\
    \|\theta_0\|_{H^N}^2&\leq \frac{1}{10} \epsilon_2^2.
  \end{align*}
  Then there exists a (maximal) time $T>0$ such that \eqref{eq:15} and
  \eqref{eq:16} hold:
  \begin{align}
    E_{\omega}(T)&\leq 8\epsilon_1^2, \\
    E_{\theta}(T) &\leq 8 \epsilon_2^2.
  \end{align}
\end{prop}

\begin{proof}
  By classical local well-posedness results for the Navier-Stokes equations (see
  \cite{bedrossian2016sobolev}) and for the Boussinesq equations (see \cite[Section
  3.5]{temam2012infinite}) for a sufficiently small time $T>0$ we obtain the
  existence of a solution $(\omega(t), \theta(t))$ with
  \begin{align*}
    \|\omega(t)\|_{H^N} \leq \epsilon_1, \|\theta(t)\|_{H^N} \leq \epsilon_2,
  \end{align*}
  for all $0< t \leq T$. Further using the dissipative structure to control
  $\nabla_t \omega$ and $\nabla_t \theta$ and possibly choosing $T$ even
  smaller, we thus may estimate $E_{\omega}(T)$ and $E_{\theta}(T)$ as claimed.
\end{proof}

Given this positive time, we next show that the estimates \eqref{eq:15}, \eqref{eq:16} actually
hold for all times.

\begin{prop}
  \label{prop:improve}
  Suppose that for $T>0$ the estimates \eqref{eq:15} and \eqref{eq:16} hold and
  let $\epsilon_1, \epsilon_2$ be as in Theorem \ref{thm:nonlinear}.
  Then the following improved estimates hold
  \begin{align}
    \label{eq:18}
    E_{\omega}(T)&\leq 4\epsilon_1^2, \\
    E_{\theta}(T) &\leq 4 \epsilon_2^2.
  \end{align}
\end{prop}

Before proving Proposition \ref{prop:improve}, let us discuss how this allows us
to establish Theorem \ref{thm:nonlinear}.

\begin{proof}[Proof of Theorem \ref{thm:nonlinear}]
  By Proposition \ref{prop:initialbootstrap} there exists a positive time $T>0$ such
  that the estimates \eqref{eq:15} and \eqref{eq:16} hold.
  Since these are closed conditions, we may take $0<T^*\leq \infty$ to be the
  maximal time such that \eqref{eq:6} holds.
  If $T^*=\infty$ this implies the results of Theorem \ref{thm:nonlinear}. Thus,
  suppose for the sake of contradiction that $T^*$ is finite.
  Then by Proposition \ref{prop:improve}, on $(0,T^*)$ the improved estimates
  \eqref{eq:18} hold. By local well-posedness and continuity arguments as in the
  proof of Proposition \ref{prop:initialbootstrap} there then exists a time $T_2>T^*$
  (possibly only very slightly larger) such that the solutions exists at least until
  time $T_2$ and the energies satisfy $E_{\omega}(T_2)-E_{\omega}(T^*)< \epsilon_1^2$ and
  $E_{\theta}(T_2)-E_{\theta}(T^*)< \epsilon_2^2$.
  But by \eqref{eq:18} this implies that that also at the larger time $T_2$, the
  estimates \eqref{eq:15} and \eqref{eq:16} are satisfied and $T^*$ is therefore not maximal. This
  contradiction thus shows that $T^*=\infty$, which concludes the proof.
\end{proof}

It thus remains to prove Proposition \ref{prop:improve}.

\begin{proof}[Proof of Proposition \ref{prop:improve}]
Let $T>0$ be a given time such that
\begin{align*}
  E_{\omega}(T)&:= \|A\omega\|_{L^\infty L^2((0,T))}^2 + \nu \|\nabla_t A \omega\|_{L^2 L^2((0,T))} + \|\sqrt{-\dot M M} \langle D \rangle^N \omega\|_{L^2 L^2((0,T))}^2 \leq 8\epsilon_1^2, \\
 E_{\theta}(T)&:= \|A\theta\|_{L^\infty L^2((0,T))}^2 + \eta \|\nabla_t A \theta\|_{L^2 L^2((0,T))} + \|\sqrt{-\dot M M} \langle D \rangle^N \theta\|_{L^2((0,T))}^2 \leq 8\epsilon_2^2. 
\end{align*}

Then by testing the Boussinesq equation \eqref{eq:B} with $A\omega$ and
$A\theta$ we observe that 
  \begin{align*}
    \dt \|A\omega\|_{L^2}^2/2 + \nu \|\nabla_t A \omega\|_{L^2}^2 + \|\sqrt{-\dot M M} \langle D \rangle^N \omega\|_{L^2}^2 &= - \int A (u \cdot \nabla \omega)A \omega + \int A(\p_x\theta ) A \omega, \\
    \dt \|A \theta\|_{L^2}^2/2 + \eta \|\nabla_t A \theta\|_{L^2}^2 + \|\sqrt{-\dot M M} \langle D \rangle^N \theta \|_{L^2}^2 &= - \int A (u \cdot \nabla \theta)A \theta - \alpha \int A(\p_x \Delta_t^{-1} \omega) A \theta. 
  \end{align*}
  Here we used that $A$ is a Fourier multiplier and hence commutes with
  derivatives, which greatly simplifies calculations (for related problems for
  flows other than Couette see \cite{WZZkolmogorov,coti2019degenerate}).

  Integrating in time, it follows that
  \begin{align*}
    E_{\omega}(T) &\leq \|A\omega_0\|_{L^2}^2 - \iint A(u \cdot \nabla \omega) A \omega + \iint A(\p_x \theta) A \omega \\
    & =: \|A\omega_0\|_{L^2}^2 + \mathcal{T}_{\omega} + \mathcal{T}_{\omega \theta},\\
      E_\theta(T)& \leq \|A\theta_0\|_{L^2}^2 - \iint A(u\cdot \nabla \theta) A \theta - \alpha \iint  A(\p_x \Delta_t^{-1} \omega) A \theta\\
    &:= \|A\theta_0\|_{L^2}^2  + \mathcal{T}_\theta + \mathcal{T}_{\alpha}.
  \end{align*}
    % In the formulation of the equations we note that
    % \begin{align*}
    %   \nabla_t^\perp \Delta_t^{-1} \omega \cdot \nabla_t \theta = \nabla^\perp \Delta_t^{-1} \omega \cdot \nabla \theta, 
    % \end{align*}
    % since the contributions by $\pm t \p_x \Delta_t^{-1} \omega \p)x \theta$
    % cancel.
    Since the initial data by assumption satisfies
    \begin{align*}
      \frac{\|A\omega_0\|_{L^2}^2}{2}&< \epsilon_1^2, \\
      \frac{\|A\theta_0\|_{L^2}^2}{2}&< \epsilon_2^2,
    \end{align*}
    it remains to estimate $\mathcal{T}_{\omega}, \mathcal{T}_{\omega,\theta}, \mathcal{T}_\theta$ and $\mathcal{T}_{\alpha}$.

    We phrase these bounds as a proposition.
    \begin{prop}
      \label{prop:estimates}
      Let $T>0$ and suppose that \eqref{eq:15} and \eqref{eq:16} hold.
      Then the following estimates hold:
      \begin{align}
        \mathcal{T}_{\omega}&\leq \epsilon_1^{3} \nu^{-1/2} + \epsilon_1^3 \nu^{-1/3}, \\
        \mathcal{T}_{\theta} &\leq \epsilon_2^2 \epsilon_1 \eta^{-1/2} + \epsilon_2^2 \epsilon_1 \nu^{-1/3}, \\
        \mathcal{T}_{\theta, \omega} &\leq \frac{\epsilon_1 \epsilon_2}{\sqrt{\nu \eta}}, \\
        \mathcal{T}_{\alpha} & \leq  \alpha \eta^{-1/2} \epsilon_2 \nu^{-1/3} \epsilon_1.
      \end{align}
    \end{prop}
These estimates allow us to conclude the proof of Proposition
\ref{prop:improve}:
Since $\epsilon_1\leq \min(\nu, \eta)^{1/2}$, $\epsilon_2\leq \sqrt{\nu \eta}
\epsilon_1$ and $\alpha< \eta^{1/2}\nu^{1/3}\frac{\epsilon_2}{\epsilon_1}$ it follows that
\begin{align*}
  \mathcal{T}_{\omega} \leq \epsilon_1^2,  \mathcal{T}_{\omega, \theta} \leq \epsilon_1^2,
  \mathcal{T}_{\theta}\leq \epsilon_2^2, \mathcal{T}_{\alpha}\leq \epsilon_2^2.
\end{align*}
This in turn implies that
\begin{align*}
  E_\omega(T)&\leq \epsilon_1^2 + \epsilon_1^2 + \epsilon_1^2 = 3 \epsilon_1^2 < 8 \epsilon_1^2, \\
  E_{\theta}(T)&\leq \epsilon_2^2 + \epsilon_2^2 +\epsilon_2^2< 8 \epsilon_2^2.
\end{align*}
Thus, we observe an improvement over the bounds \eqref{eq:15} and \eqref{eq:16},
which concludes the proof of this proposition and hence allows us to close the
bootstrap argument for Theorem \ref{thm:nonlinear}.
\end{proof}

It remains to prove Proposition \ref{prop:estimates}.
\begin{proof}[Proof of Proposition \ref{prop:estimates}]
  We remark that $\mathcal{T}_{\theta,\omega}$ and $\mathcal{T}_\alpha$ have a quadratic structure as
  opposed to the cubic structure of $\mathcal{T}_{\theta}$ and
  $\mathcal{T}_{\omega}$.
  Hence, the additional smallness compared to $8\epsilon_1^2$ or $8
  \epsilon_2^2$ in these two cases is achieved
  by requiring that $\epsilon_2$ is much smaller than $\epsilon_1$ and that
  $\alpha$ is small compared to the quotient $\frac{\epsilon_2}{\epsilon_1}$.

  \underline{Estimating $\mathcal{T}_{\omega,\theta}$:}
  Since $\p_x \theta$ possesses a vanishing $x$-average, we may use Hölder's
  inequality and Poincar\'e's inequality to estimate
  \begin{align*}
    \mathcal{T}_{\theta,\omega} = \int_0^T \langle A\omega_{\neq}, A \p_x \theta \rangle
    \leq \|A\omega_{\neq}\|_{L^2 L^2} \|\p_x A \theta\|_{L^2 L^2}\\
    \leq \|\nabla_t A\omega_{\neq}\|_{L^2 L^2} \|\nabla_t A \theta\|_{L^2 L^2} \\
    \leq \frac{\epsilon_1}{\sqrt{\nu}} \frac{\epsilon_2}{\sqrt{\eta}}.
  \end{align*}

\underline{Estimating $\mathcal{T}_\alpha$:}
We recall that 
  \begin{align*}
    \mathcal{T}_{\alpha}= \alpha \int_0^T \langle A  \theta, A \p_x \Delta_{t}^{-1} \omega \rangle dt.
  \end{align*}
  Using Hölder's and Poincar\'e's inequality we control this by
  \begin{align*}
    & \quad \alpha \| \nabla A \theta\|_{L^2 L^2} \|\p_x \Delta_t^{-1} \omega\|_{L^2 H^N} \\
    &\leq \alpha \eta^{-1/2} \epsilon_2 \nu^{-1/3} \epsilon_1.
  \end{align*}
  We remark that here is where we use that $\alpha$ is ``small''. An alternative
  approach for $\alpha$ ``large'' is discussed in Section \ref{sec:largealpha}.

\underline{Estimating $\mathcal{T}_{\omega}$ and $\mathcal{T}_{\theta}$:}
  The estimate for $\mathcal{T}_{\omega}$ has been established in
  \cite{bedrossian2016sobolev}. Its proof further extends to the case of
  $\mathcal{T}_{\theta}$ with minor modifications.
  In the interest of readability we include it below.
  
  We recall that
  \begin{align*}
    \mathcal{T}_\omega= - \iint_0^T A(v \cdot \nabla \omega) A \omega
  \end{align*}
  Since the shear flow component of the velocity field, that is the $x$-average $(\nabla_t^\perp
  \Delta_t^{-1}\omega)_{=}= \p_y^{-1} \omega_{=}=:v_=$, decays slower, we split
  $\mathcal{T}_{\omega}$ into a contribution involving the shear and a
  contribution involving its $L^2$-orthogonal complement:
  \begin{align*}
    \mathcal{T}_{\omega}= \int_0^T \langle A\omega, A(v_{=} \p_x \omega) \rangle
    + \int_0^T \langle A\omega, A(v_{\neq} \cdot \nabla \omega) \rangle
    = \mathcal{T}_{\omega}^{=} + \mathcal{T}_{\omega}^{\neq}
  \end{align*}
  For $\mathcal{T}_{\omega}^{\neq}$ we easily estimate by
  \begin{align}
    \label{eq:19}
    \begin{split}
    \|A \omega\|_{L^\infty L^2}\|v_{\neq} \|_{L^2 H^N} \|\nabla \omega\|_{L^2 H^N}\\
    \leq \epsilon_1 \ \epsilon_1 \ \frac{\epsilon_1}{\sqrt{\nu}} = \epsilon_1^3 \nu^{-1/2}.
    \end{split}
  \end{align}
  In order to estimate $\mathcal{T}_{\omega}^=$ we make use of some
  cancellations. We note that $\p_x v_==0$ and hence
  \begin{align*}
    \langle A\omega, v_= \p_x A \omega \rangle=0.
  \end{align*}
  We therefore obtain a commutator
  \begin{align*}
    \mathcal{T}_{\omega}^{=} =\int_0^T \langle A \omega, (A(v_{=} \p_x \omega_{\neq})- u_0 \p_x A \omega_{\neq}) \rangle dt.
  \end{align*}
  By Parseval's theorem we express the inner $L^2$ integral as
  \begin{align*}
    \sum_k \iint a(t,k,\xi) \tilde{\omega}(k,\xi) (a(k,\xi)-a(k,\xi-\zeta)) \tilde{u}_=(\zeta) \tilde{\omega}(k,\xi-\zeta) d\zeta d\xi
  \end{align*}
  By the properties of $A$ (and $M$) Bedrossian, Vicol and Wang deduce (see $(2.17)$ and $(2.18)$ in
  \cite{bedrossian2016sobolev}) that
  \begin{align*}
     |A(k,\xi)-A(k,\xi-\zeta)| \leq ((1+k^2+ (\xi-\zeta)^2)^{N/2} + (1+k^2+\xi^2)^{N/2})|zeta|.
  \end{align*}
  We note that the factor $|\zeta|$ cancels with $\tilde{v}_=(\zeta)= -i
  \zeta^{-1} \tilde{\omega}(0,\zeta)$ and hence obtain that 
  \begin{align*}
    |\mathcal{T}_{\omega}^=|\leq C \sum_{k\neq 0} \iint ((1+k^2+(\xi-\zeta)^2)^{N/2} + (1+k^2+\zeta^2)^{N/2})
    |\tilde{\omega}(0,\zeta) \tilde{\omega}(k,\xi-\zeta)| |A(k,\xi)\tilde{\omega}(k,\xi)| d\xi d\zeta.
  \end{align*}
  It thus follows that
  \begin{align*}
  |\mathcal{T}_\omega^=|\leq C  \|\omega_{=}\|_{L^\infty H^N} \|\omega_{\neq}\|_{L^2H^N}^2.
  \end{align*}
  As noted in \eqref{eq:17} following the introduction of the multiplier $M$, the
  last term can be estimate in terms of $\nu^{-1/3} E_{\omega}$ and therefore
  \begin{align}
    \label{eq:20}
   \mathcal{T}_{\omega}^{=} \leq  \epsilon_1 \nu^{-1/3} \epsilon_1^2.
  \end{align}
  Combining the estimate \eqref{eq:19} for $\mathcal{T}_\omega^{\neq}$ and
  \eqref{eq:20} for $\mathcal{T}_\omega^=$ then concludes the proof for $\mathcal{T}_\omega$.

  We next consider $\mathcal{T}_{\theta}$ and analogously split into a
  contribution involving the shear and one involving its complement:
  \begin{align*}
    \mathcal{T}_{\theta}&= \int_0^T \langle A \theta_{\neq} , (A (u_= \p_x \theta)- u_= \p_x A \theta_\neq) \rangle + \int_0^T \langle A \theta, A(u_\neq \cdot \nabla \theta) \rangle \\
    &=: \mathcal{T}_{\theta}^{\neq} + \mathcal{T}_{\theta}^=.
    % \leq C \|\omega_=\|_{L^\infty H^N} \|\theta_\neq \|_{L^2 H^N}^2
    % + \|\nabla_t \Delta_t^{-1} \omega_\neq \|_{L^2 H^N} \|\nabla_t \theta\|_{L^2 H^N} \|A \theta\|_{L^\infty H^N} \\
    % \leq \epsilon_1 \eta^{-1/2} \nu^{-1/3}\epsilon_2^2 + \epsilon_2^2 \epsilon_1 \nu^{-1/3}
  \end{align*}
  By the same argument as for $\mathcal{T}_{\omega}$ we may estimate
  \begin{align*}
    |\mathcal{T}_\theta^{\neq}| \leq C \|\omega_=\|_{L^\infty H^N} \|\theta_{\neq} \|_{L^2 H^N}^2 \leq \epsilon_1 \eta^{-1/2} \nu^{-1/3}\epsilon_2^2
  \end{align*}
  and
  \begin{align*}
    |\mathcal{T}_{\theta}^=| \leq \|\nabla_t \Delta_t^{-1} \omega_{\neq} \|_{L^2 H^N} \|\nabla_t \theta\|_{L^2 H^N} \|A  \theta\|_{L^\infty H^N} \leq \epsilon_2^2 \epsilon_1 \nu^{-1/3}.
  \end{align*}
  This concludes the proof.
\end{proof}
\subsection{Large Hydrostatic Balance and Shear}
\label{sec:largealpha}
In Section \ref{sec:fulldissipation} we considered the nonlinear problem with
$\alpha>0$ ``small'' as a perturbation of the Navier-Stokes problem.
In contrast in Section \ref{sec:alpha} for the linearized problem we exploited
$\alpha$ to make use of classical energy methods used for the hydrostatic
balance case (without shear) and treated the shear $\beta y$ as a correction.
Our aim in the following is to combine both methods to establish (asymptotic)
stability also for large $\alpha$ and $\beta=1$ (after rescaling).

Here, we further adapt the previous bootstrap approach to consider an energy of the
form 
\begin{align}
  \label{eq:21}
  \alpha \|A \omega \|_{H^{N}}^2 + \langle A \theta, (-\p_x^2-(\p_y-t\p_x)^2) A \theta\rangle_{L^2} .
\end{align}

\begin{thm}
  \label{thm:largealpha}
  Let $\alpha\geq 1$ and $\beta=1$ and $\nu>0$ and suppose that $\eta>2$.
  Let further $(\omega_0, \theta_0)\in H^N \times H^{N+1}$ be given initial data
  such that
  \begin{align*}
    \alpha \|\omega_0\|_{H^N}^2 + \|\nabla \theta_0\|_{H^N}^2 \ll \epsilon^2.
  \end{align*}
  Then for all times $T>0$ it holds that
  \begin{align}
    \begin{split}
    \text{ess-sup}_{0\leq t \leq T} (\alpha \|A \omega(t) \|_{L^2}^2 +  \langle A \theta, (-\p_x^2-(\p_y-t\p_x)^2) A \theta\rangle_{L^2} )\\
    + \nu \int_0^T \alpha \| \nabla_t A \omega(t) \|_{H^{N}}^2 dt \\
    \label{eq:third} + \eta\int_0^T \langle A \theta, (-\p_x^2-(\p_y-t\p_x)^2)^2 A \theta\rangle_{L^2} dt \leq \epsilon^2.
    \end{split}
  \end{align}
\end{thm}
We remark that lower bound on $\eta$ is very restrictive, but allows use to treat
the time-dependence of $(\p_y-t \p_x)^2$ perturbatively.
In the general case $\beta \in \R$ this restriction would read $\eta \gg \beta$
and thus requires that thermal dissipation dominates the shear. 

\begin{proof}[Proof of Theorem \ref{thm:largealpha}]
  Similarly to the proof of Theorem \ref{thm:nonlinear} we begin by considering
  the time-derivative of equation \eqref{eq:21}.
  We compute
  \begin{align*}
    &\quad  \frac{d}{dt} \alpha \|A \omega \|_{H^{N}}^2+ \|\sqrt{-\dot A A} \omega\|_{L^2} + \nu \|\nabla_t A \omega\|_{L^2} \\
    &= \alpha \langle A \omega, A (v \cdot \nabla \omega) \rangle \\
    & \quad + \alpha \langle A \omega, A \p_x \theta \rangle,
  \end{align*}
  and
  \begin{align*}
    & \quad \frac{d}{dt} \langle A \theta, (-\p_x^2-(\p_y-t\p_x)^2) A \theta\rangle_{L^2} + \|\sqrt{-\p_x^2-(\p_y-t\p_x)^2}\sqrt{-\dot A A} \theta\|_{L^2}^2 \\
    & \quad + \eta \|\nabla_t \sqrt{-\p_x^2-(\p_y-t\p_x)^2} A \theta\|_{L^2}^2 \\
    &=  \alpha \langle A \theta ,(-\p_x^2-(\p_y-t\p_x)^2) A \p_x (-\p_x^2-(\p_y-t\p_x)^2)^{-1} \omega \rangle \\
    & \quad + \langle A \theta, (-\p_x^2- (\p_y-t\p_x)^2) A (v \cdot \nabla_t \theta) \rangle \\
    & \quad + \langle A \theta, -2\p_x (\p_y-t\p_x) A \theta \rangle.
  \end{align*}
  Since $A$ is a Fourier multiplier and hence commutes with $(-\p_x^2-(\p_y-t\p_x)^2)$, we observe that the contributions
  \begin{align*}
    \alpha \langle A \omega, A \p_x \theta \rangle
  \end{align*}
  and
  \begin{align*}
    \alpha \langle A \theta ,(-\p_x^2-(\p_y-t\p_x)^2) A \p_x (-\p_x^2-(\p_y-t\p_x)^2)^{-1} \omega \rangle
  \end{align*}
  cancel out.

  Integrating from $0$ to $T$ as in the proof of Theorem \ref{thm:nonlinear}, in
  our bootstrap approach we thus have to control three contributions:
  \begin{align}
   \mathcal{T}_{\omega}:= \int_0^T \alpha \langle A \omega, A (v \cdot \nabla \omega) \rangle, \\
    \mathcal{T}_{\theta}:= \int_0^T \langle A \theta, (-\p_x^2- (\p_y-t\p_x)^2) A (v \cdot \nabla_t \theta) \rangle,
  \end{align}
  and
  \begin{align*}
    \mathcal{T}_{A}:=\int_0^T \langle A \theta, -2\p_x (\p_y-t\p_x) A \theta \rangle.
  \end{align*}
  The first contribution $\mathcal{T}_{\omega}$ can be controlled in exactly the same way as in the
  proof of Proposition \ref{prop:estimates}:
  \begin{align*}
    \mathcal{T}_\omega \leq \alpha \|A \omega\|_{L^\infty L^2} \|u_{\neq}\|_{L^2H^N} \|\nabla \omega\|_{L^2 H^N}
    + C \alpha \|\omega_{=}\|_{L^\infty H^N} \|\omega_{\neq}\|_{L^2 H^N}^2 \\
    \leq  \epsilon^3 (\frac{\nu^{-1/2}}{\sqrt{\alpha}} + \frac{\nu^{-1/3}}{\sqrt{\alpha}}).
  \end{align*}
  The contribution $\mathcal{T}_A$ can be absorbed into 
  \begin{align*}
     \eta \|\nabla_t \sqrt{-\p_x^2-(\p_y-t\p_x)^2} A \theta\|_{L^2 L^2}^2 
  \end{align*}
  by using that $\eta\geq 2$.

  Finally, for the contribution $\mathcal{T}_{\theta}$ we follow the same
  strategy of proof as in Proposition \ref{prop:estimates}.
  We again split $\mathcal{T}_{\theta}$ into contributions due to $v_{=}$ and $v_{\neq}$.
  For 
  \begin{align*}
    \mathcal{T}_{\theta}^{\neq} = \int_0^T \int_0^T \langle A \theta, (-\p_x^2- (\p_y-t\p_x)^2) A (v_{\neq} \cdot \nabla_t \theta) \rangle
  \end{align*}
  we may estimate by
  \begin{align*}
    \|\sqrt{-\p_x^2- (\p_y-t\p_x)^2}A \theta\|_{L^\infty L^2} \|v_{\neq}\|_{L^2 H^N} \|\nabla_t\sqrt{-\p_x^2- (\p_y-t\p_x)^2} \theta\|_{L^2 H^N} \\
    + \|\sqrt{-\p_x^2- (\p_y-t\p_x)^2}A \theta\|_{L^\infty L^2} \|\omega_{\neq}\|_{L^2 H^N} \|\sqrt{-\p_x^2- (\p_y-t\p_x)^2} \theta\|_{L^2 H^N} \\
    \leq \epsilon \epsilon \nu^{-1/3} \epsilon \eta^{-1/2} + \epsilon \epsilon \nu^{-1/2} \epsilon \eta^{-1/2}.
  \end{align*}
  Compared to the setting of Theorem \ref{thm:nonlinear} we thus lose more
  powers of $\nu$ and $\eta$.
  
  For the last contribution
  \begin{align*}
    \mathcal{T}_{\theta}^{=} = \int_0^T \int_0^T \langle A \theta_{\neq}, (-\p_x^2- (\p_y-t\p_x)^2) A (v_{=}\p_x \theta_{\neq}) \rangle,
  \end{align*}
  we again use Parseval's theorem to obtain a cancellation for the contributions
  by $\theta_{=}$.
  Next, we integrate $-\p_x^2- (\p_y-t\p_x)^2$ by parts once and use the product
  rule to split
  \begin{align*}
    \p_x A (v_{=}\p_x \theta_{\neq}) &= A ((\p_x v_{\neq})\p_x \theta) + A (v_{\neq} \p_x \p_x \theta),\\
    (\p_y-t\p_x) A (v_{=}\p_x \theta_{\neq}) &= A (((\p_y-t\p_x) v_{\neq})\p_x \theta) + A (v_{\neq} \p_x (\p_y-t\p_x) \theta).
  \end{align*}
  For the first terms we bound by
  \begin{align*}
    \|\sqrt{-\p_x^2- (\p_y-t\p_x)^2} A \theta \|_{L^2 L^2} \|\omega\|_{L^\infty H^N} \|\nabla_t \theta \|_{L^2 H^N} \\
    \leq \epsilon \frac{\epsilon}{\sqrt{\alpha}} \epsilon.
  \end{align*}
  For the second terms we argue exactly as in the proof of Proposition
  \ref{prop:estimates} with $\p_x \theta$ or $(\p_y-t\p_y) \theta$ in place of
  $\theta$, which yields a bound by
  \begin{align*}
    \|\omega_{=}\|_{L^\infty H^N} \|\sqrt{\p_x^2+(\p_y-t\p_x)^2}\theta_{\neq}\|_{L^2 H^N}^2\\
    \leq \epsilon^3 \nu^{-1/2} \nu^{-1/3}
  \end{align*}
\end{proof}

\section{From Bounds to Decay}
\label{sec:bounds}

In Theorem \ref{thm:nonlinear} in Section \ref{sec:fulldissipation} we have
shown that the nonlinear Boussinesq equations satisfy energy estimates of the form 
\begin{align}
  \label{eq:22}
  \omega, \theta \in L^\infty_t H^N, \\
  \label{eq:23}
  \nabla_t \omega, \nabla_t \theta \in L^2_t H^N.
\end{align}
Hence, we know that the solutions stay bounded and their gradients are
integrable in time. However, integrability does not by itself imply any decay
(consider for example a series of thinner and thinner step functions) and even if one
additionally requires uniform continuity it only implies convergence to zero but
yields no rate.

In the following we make use of additional bounds on the semigroup associated with the
linearized operator to deduce decay estimates.
\begin{prop}
  \label{prop:decayrates}
  Let $N, \alpha, \epsilon_1, \epsilon_2$ be as in Theorem \ref{thm:nonlinear}.
  Additionally suppose you know the following two estimates:
\begin{itemize}
\item The evolution semigroup $S(\cdot , \cdot)$ of the linearized problem satisfies
  \begin{align*}
    \|S(t,\tau)\|_{H^N \times H^{N+1} \rightarrow H^N \times H^{N+1}} \leq C \exp(-C \gamma (t-\tau))
  \end{align*}
  for any $t\geq \tau \geq 0$ and some $\gamma>0$. (This is established in
  Section \ref{sec:couette}).
\item Due to (enhanced) dissipation $\omega \in L^2_tH^{N+1}$ and we have the
  following estimate:
  \begin{align*}
    \|\omega\|_{L^2 H^{N+1}} \ll \frac{1}{\alpha}
  \end{align*}
  (This follows by Theorem \ref{thm:nonlinear})
\end{itemize}
  Then the nonlinear Boussinesq equations further satisfy
  \begin{align*}
    \|\omega(t)\|_{H^N} + \|\theta(t)\|_{H^N} \leq 2 C \exp(-C\gamma t/2) (\|\omega_0\|_{H^N} + \|\theta_0\|_{H^N})
  \end{align*}
  for all $t>0$.
  In particular, we may choose $\gamma= \min(\nu, \eta)^{1/3}$ and thus observe
  dissipation on a time scale faster than heat flow, that is \emph{enhanced dissipation}.
\end{prop}

We remark that the linearized problem around Couette flow decays with a
rate $\exp(-C (\nu^{1/3}t)^3)$ (see Section \ref{sec:homogen}), which we may
estimate from above by
\begin{align*}
  \exp(C) \exp(-C \nu^{1/3}t),
\end{align*}
since $t^3\geq t-1$ for all $t\geq 0$.
To the author's knowledge it is not known whether the nonlinear Navier-Stokes
problem exhibits the same faster $\exp(-C(\nu^{1/3}t)^3)$ decay instead of the
exponential decay by $\exp(-C\nu^{1/3}t)$.

\begin{proof}[Proof of Proposition \ref{prop:decayrates}]
  In order to prove Proposition \ref{prop:decayrates} we again use a bootstrap
  approach.
  For this purpose we interpret the nonlinear problem as a
forced linear problem:
\begin{align*}
  \dt \omega + \nu \Delta_t \omega + \p_x \theta = - u \cdot \nabla_t \omega =: f, \\
  \dt \theta + \eta \Delta_t \theta = - u \cdot \nabla \theta =:g.
\end{align*}
Denoting the semigroup of the linearized problem by $S(\cdot, \cdot)$, we obtain
the integral equation
\begin{align}
  \label{eq:24}
  (\omega(t), \theta(t))= S(t,0) (\omega_0, \theta_0) + \int_0^t S(t, \tau) (f(\tau), g(\tau)) d\tau.
\end{align}
By assumption on the decay rate of the semi-group the first contribution can be
estimated by
\begin{align*}
  \exp(-\gamma t) (\|\omega_0\|_{H^N} + \|\theta_0\|_{H^N}).
\end{align*}
For the nonlinear contribution we derive a first, rough estimate by using that
\begin{align}
  \label{eq:25}
  \|f\|_{H^N} \leq \|\omega\|_{H^N} \|\nabla_t \omega\|_{H^N}, \\
  \|g\|_{H^{N+1}} \leq \|\omega\|_{H^N} \|\nabla_t \theta\|_{H^{N+1}}.
\end{align}
It then follows that at least for very small times the nonlinear contribution is
bounded by $\epsilon^2$ and as a consequence for these small times
\begin{align}
  \label{eq:26}
  \|(\omega(t), \theta(t))\|_{H^N}  \leq 2 \exp(-\gamma t/2) \epsilon.
\end{align}
We next argue by a bootstrap iteration that the estimate \eqref{eq:26} holds for
all times.
Thus suppose that \eqref{eq:26} holds for $0\leq t \leq T$ and assume for the
sake of contradiction that $T<\infty$ is maximal.
The first contribution in \eqref{eq:24} is bounded by
\begin{align*}
  \exp(-\gamma t) \epsilon
\end{align*}
and thus both small and fast decaying.
We hence focus on the contribution by the nonlinearity.
Here we combine the combine the decay estimate of $S(t,\tau)$, \eqref{eq:25} and
\eqref{eq:26} to estimate
\begin{align*}
  \left\| \int_0^t S(t, \tau) (f(\tau), g(\tau)) d\tau \right\| d\tau \\
  \leq \int_0^t C \exp(-\gamma (t-\tau)) (\|\omega(\tau)\|_{H^N} \|\nabla_t \omega(\tau)\|_{H^N} + \|\omega(\tau)\|_{H^N} \|\nabla_t \theta(\tau)\|_{H^N}) d\tau \\
  \leq C^2 \int_0^t \exp(-\gamma(t-\tau)) \exp(-\gamma \tau /2) (\|\nabla_t \omega(\tau)\|_{H^N} +\|\nabla_t \theta(\tau)\|_{H^N}) d\tau \\
  \leq C^2 \epsilon \exp(-\gamma t /2) \int_0^t\exp(-\gamma(t-\tau)/2) (\|\nabla_t \omega(\tau)\|_{H^N} +\|\nabla_t \theta(\tau)\|_{H^N}) d\tau.
\end{align*}
We then use the $L^2$ integrability assumption on $(\|\nabla_t
\omega(\tau)\|_{H^N} +\|\nabla_t \theta(\tau)\|_{H^N})$ and the Cauchy-Schwarz
inequality to further bound this by
\begin{align*}
  C^2 \frac{\epsilon^2 (\nu^{-1/2}+\eta^{-1/2})}{\sqrt{\gamma}}\exp(-\gamma t /2).
\end{align*}
By the assumption on $\epsilon$ this is smaller than
\begin{align*}
  \epsilon \exp(-\gamma t /2) < 2 \epsilon \exp(-\gamma t /2).
\end{align*}
Thus equality in \eqref{eq:26} is not attained for $t=T$, which contradicts the
maximality of $T$. Therefore, the maximal time is infinity, which concludes the proof.
\end{proof}

\subsection*{Acknowledgments}
Christian Zillinger's research is supported by the ERCEA under the grant 014
669689-HADE and also by the Basque Government through the BERC 2014-2017
program and by Spanish Ministry of Economy and Competitiveness MINECO: BCAM Severo Ochoa excellence accreditation SEV-2013-0323.

\bibliographystyle{alpha}
\bibliography{citations2}

\begin{thebibliography}{DWZZ18}

\bibitem[BM15]{bedrossian2013inviscid}
Jacob Bedrossian and Nader Masmoudi.
\newblock Inviscid damping and the asymptotic stability of planar shear flows
  in the 2d {E}uler equations.
\newblock {\em Publications math{\'e}matiques de l'IH{\'E}S}, 122(1):195--300,
  2015.

\bibitem[BMM16]{bedrossian2013landau}
Jacob Bedrossian, Nader Masmoudi, and Cl{\'e}ment Mouhot.
\newblock Landau damping: paraproducts and {G}evrey regularity.
\newblock {\em Annals of PDE}, 2(1):4, 2016.

\bibitem[BMV16]{bedrossian2016enhanced}
Jacob Bedrossian, Nader Masmoudi, and Vlad Vicol.
\newblock Enhanced dissipation and inviscid damping in the inviscid limit of
  the {N}avier--{S}tokes equations near the two dimensional {C}ouette flow.
\newblock {\em Archive for Rational Mechanics and Analysis}, 219(3):1087--1159,
  2016.

\bibitem[BVW16]{bedrossian2016sobolev}
Jacob Bedrossian, Vlad Vicol, and Fei Wang.
\newblock The {S}obolev stability threshold for {2D} shear flows near
  {C}ouette.
\newblock {\em arXiv preprint arXiv:1604.01831}, 2016.

\bibitem[CD80]{cannon1980initial}
JR~Cannon and Emmanuele DiBenedetto.
\newblock The initial value problem for the {B}oussinesq equations with data in
  {$L^p$}.
\newblock In {\em Approximation methods for Navier-Stokes problems}, pages
  129--144. Springer, 1980.

\bibitem[Cha06]{chae2006global}
Dongho Chae.
\newblock Global regularity for the 2d {B}oussinesq equations with partial
  viscosity terms.
\newblock {\em Advances in Mathematics}, 203(2):497--513, 2006.

\bibitem[CKN99]{chae1999local}
Dongho Chae, Sung-Ki Kim, and Hee-Seok Nam.
\newblock Local existence and blow-up criterion of {H}{\"o}lder continuous
  solutions of the {B}oussinesq equations.
\newblock {\em Nagoya Mathematical Journal}, 155:55--80, 1999.

\bibitem[CZZ19]{coti2019degenerate}
Michele Coti~Zelati and Christian Zillinger.
\newblock On degenerate circular and shear flows: the point vortex and power
  law circular flows.
\newblock {\em Communications in Partial Differential Equations},
  44(2):110--155, 2019.

\bibitem[{\relax DLMF}]{NIST:DLMF}
{\it NIST Digital Library of Mathematical Functions}.
\newblock http://dlmf.nist.gov/, Release 1.0.24 of 2019-09-15.
\newblock F.~W.~J. Olver, A.~B. {Olde Daalhuis}, D.~W. Lozier, B.~I. Schneider,
  R.~F. Boisvert, C.~W. Clark, B.~R. Miller, B.~V. Saunders, H.~S. Cohl, and
  M.~A. McClain, eds.

\bibitem[DM18]{dengmasmoudi2018}
Yu~Deng and Nader Masmoudi.
\newblock Long time instability of the {C}ouette flow in low {G}evrey spaces.
\newblock {\em arXiv preprint arXiv:1803.01246}, 2018.

\bibitem[DWZZ18]{doering2018long}
Charles~R Doering, Jiahong Wu, Kun Zhao, and Xiaoming Zheng.
\newblock Long time behavior of the two-dimensional {B}oussinesq equations
  without buoyancy diffusion.
\newblock {\em Physica D: Nonlinear Phenomena}, 376:144--159, 2018.

\bibitem[DZ19]{dengZ2019}
Yu~Deng and Christian Zillinger.
\newblock Echo chains as a linear mechanism: Norm inflation, modified exponents
  and asymptotics.
\newblock {\em arXiv preprint arXiv:1910.12914}, 2019.

\bibitem[FMT87]{foias1987attractors}
C~Foias, O~Manley, and R~Temam.
\newblock Attractors for the b{\'e}nard problem: existence and physical bounds
  on their fractal dimension.
\newblock {\em Nonlinear Analysis: Theory, Methods \& Applications},
  11(8):939--967, 1987.

\bibitem[Jia19]{jia2019linear}
Hao Jia.
\newblock Linear inviscid damping in {G}evrey spaces.
\newblock {\em arXiv preprint arXiv:1904.01188}, 2019.

\bibitem[LLT13]{larios2013global}
Adam Larios, Evelyn Lunasin, and Edriss~S Titi.
\newblock Global well-posedness for the 2d {B}oussinesq system with anisotropic
  viscosity and without heat diffusion.
\newblock {\em Journal of Differential Equations}, 255(9):2636--2654, 2013.

\bibitem[LT16]{li2016global}
Jinkai Li and Edriss~S Titi.
\newblock Global well-posedness of the 2d {B}oussinesq equations with vertical
  dissipation.
\newblock {\em Archive for Rational Mechanics and Analysis}, 220(3):983--1001,
  2016.

\bibitem[Luo]{luosobolev}
Xiang Luo.
\newblock The {S}obolev stability threshold of 2d hyperviscosity equations for
  shear flows near {C}ouette flow.
\newblock {\em Mathematical Methods in the Applied Sciences}.

\bibitem[Tem12]{temam2012infinite}
Roger Temam.
\newblock {\em Infinite-dimensional dynamical systems in mechanics and
  physics}, volume~68.
\newblock Springer Science \& Business Media, 2012.

\bibitem[TW19]{tao20192d}
Lizheng Tao and Jiahong Wu.
\newblock The 2d {B}oussinesq equations with vertical dissipation and linear
  stability of shear flows.
\newblock {\em Journal of Differential Equations}, 267(3):1731--1747, 2019.

\bibitem[Wu12]{wu20122d}
JH~Wu.
\newblock The 2d incompressible {B}oussinesq equations.
\newblock {\em Peking University Summer School Lecture Notes, Beijing}, 2012.

\bibitem[WXZ19]{wu2019stability}
Jiahong Wu, Xiaojing Xu, and Ning Zhu.
\newblock Stability and decay rates for a variant of the 2d
  {B}oussinesq--{B}{\'e}nard system.
\newblock {\em Communications in Mathematical Sciences}, 17(8):2325--2352,
  2019.

\bibitem[WZZ17]{WZZkolmogorov}
D.~{Wei}, Z.~{Zhang}, and W.~{Zhao}.
\newblock {Linear inviscid damping and enhanced dissipation for the Kolmogorov
  flow}.
\newblock {\em ArXiv e-prints}, November 2017.

\bibitem[YL18]{yang2018linear}
Jincheng Yang and Zhiwu Lin.
\newblock Linear inviscid damping for {C}ouette flow in stratified fluid.
\newblock {\em Journal of Mathematical Fluid Mechanics}, 20(2):445--472, 2018.

\bibitem[Zil19]{zillinger2019linear}
Christian Zillinger.
\newblock Linear inviscid damping in {S}obolev and {G}evrey spaces.
\newblock {\em arXiv preprint arXiv:1911.00880}, 2019.

\bibitem[Zil20]{zillinger2020landau}
Christian Zillinger.
\newblock On echo chains in {L}andau damping: Self-similar solutions and
  {G}evrey 3 as a linear stability threshold.
\newblock {\em arXiv preprint arXiv:2001.00513}, 2020.

\end{thebibliography}

\end{document}